\documentclass[11pt]{article}
\newcommand{\version}{July. 16, 2025 }

\usepackage{amsthm,amsfonts,amsmath, amscd,dsfont}
\swapnumbers
\theoremstyle{plain}

\newtheorem{thm}{THEOREM}[section]
\newtheorem{lm}[thm]{LEMMA}

\theoremstyle{definition}
\newtheorem{defi}[thm]{DEFINITION}
\theoremstyle{definition}
\newtheorem{remark}[thm]{Remark}
\newcommand{\upchi}{\raise1pt\hbox{$\chi$}}
\newcommand{\R}{{\mathord{\mathbb R}}}

\newcommand{\N}{{\mathord{\mathbb N}}}

\newcommand{\cA}{{\mathord{\cal A}}}

\newcommand{\K}{{\mathcal{K}}}

\newcommand{\cH}{{\mathcal{H}}}

\renewcommand{\|}{{\Vert}}
\numberwithin{equation}{section}
\def\E{{\cal E}_{\gamma,N,E}}

\def\dd{{\rm d}}
\def\ncht{\left(\begin{matrix} N\cr 2\cr \end{matrix}\right)}
\def\nmcht{\left(\begin{matrix} N-1\cr 2\cr \end{matrix}\right)}
\def\ST{{\mathcal S}_N}
\def\STE{{\mathcal S}_{N,E}}

\def\LN{L_{\gamma,N,E}}

\def\bset{{\boldsymbol \eta}}
\def\scF{{\mathcal F}}

\usepackage{color}

\newcommand{\one}{{\mathds1}}
\newcommand{\Dto}{\widetilde{\mathcal{G}}_{\gamma,N}}

\begin{document}

\markboth{\scriptsize{CPT \version}}{\scriptsize{CPT \version}}

\title{Spectral Gap for the Stochastic Exchange Model}

\author{\vspace{5pt} Eric A. Carlen$^{1}$,  Gustavo Posta$^{2}$,  and Imre P\'eter T\'oth$^{3}$ \\
\vspace{5pt}\small{$1.$ Department of Mathematics, Hill Center, Rutgers University}\\[-6pt]
\small{
110 Frelinghuysen Road
Piscataway NJ 08854 USA}\\
\vspace{5pt}\small{$2.$ Dipartimento di Matematica G. Castelnuovo, Sapienza Universit\`a di Roma,}
\\[-6pt] \small{Piazzale Aldo Moro, 5,
I-00185 Roma, Italy}\\
\vspace{5pt}\small{$3.$  Department of Stochastics, TU Budapest,} \\[-6pt]
\small{Egry József u. 1, H-1111, Budapest, Hungary}\\
 }
\date{\today}

\maketitle

\def\dd{{\rm d}}
\def\ncht{\left(\begin{matrix} N\cr 2\cr \end{matrix}\right)}
\def\nmcht{\left(\begin{matrix} N-1\cr 2\cr \end{matrix}\right)}

\begin{abstract} 
We prove a spectral gap inequality for the stochastic exchange model studied by Gaspard and Gilbert and by Grigo, Khanin and Sz\'asz in connection with understanding heat conduction in a deterministic billiards model. The bound on the spectral gap that we prove  is uniform in the number of particles, as had been conjectured. We adapt techniques that were originally developed to prove spectral gap bounds for the Kac model with hard sphere collisions, which, like the stochastic exchange model, has degenerate jump rates.

\end{abstract}

\medskip
\leftline{\footnotesize{\qquad Mathematics subject classification numbers:  60J25, 60J46, 45C05}}
\leftline{\footnotesize{\qquad Key Words: spectral gap, Markov jump process, degenerate rates}}

\section{Introduction} \label{intro}

\subsection{The stochastic exchange model}

The stochastic exchange model was studied in  \cite{GG08,GG09} and in \cite{GKS12} in order to understand how a deterministic billiards model can give rise to a diffusive law for heat conduction. It was also later studied by Sasada \cite{S15}. In this earlier work, the problem of proving the relevant  spectral gap bound for a stochastic exchange model with nearest-neighbor interactions was successfully reduced to the problem of proving a spectral gap bound for a mean-field model that is uniform in $N$, the number of particles \cite{S15}. However, a complete proof of this last bound had been missing and it is provided here. 

For $N\in \N$ and $E > 0$,  let $\STE$ be the set
consisting of all vectors   $\bset = (\eta_1,\dots , \eta_N)\in \R_+^N$  with
$$\frac1N\sum_{j=1}^N\eta_j = E \\ .$$
A point $\bset\in \STE$  specifies the energies of a collection of $N$ particles, and 
$\STE$ is the {\em state space} for $N$ particles with mean energy $E$ per particle.  The stochastic exchange model that we consider  is a continuous time random walk on $\STE$
 in which each step of the walk corresponds to a binary collision of a single pair of particles.
 With each collision, 
 the state of the process ``jumps'' from $ (\eta_1,\dots , \eta_N)$ to
$$ (\eta_1,\eta_2,\dots,\eta_i^*,\dots,\eta_j^*,\dots,\eta_N)\ .$$
In the stochastic exchange model, the two active particles repartition their total energy $\eta_i + \eta_j$ into two new energies $\eta_i^*$ and $\eta_j^*$, dividing the total energy uniformly at random -- the stochastic exchange. Then of course 
\begin{equation}\label{kinpos}
\eta_i^*+\eta_j^* = \eta_i+\eta_j\ .
\end{equation}

The configuration of the energies after the stochastic exchange of energy between particles $i$ and $j$ will be $\bset^*_{i,j,\alpha}$, for some randomly chosen $\alpha\in[0,1]$ where
\begin{equation}\label{POSTCOL}
(\bset^*_{i,j,\alpha})_k = \begin{cases} (1-\alpha)\eta_i+\alpha \eta_j & k = i\\   \alpha \eta_i+(1-\alpha)\eta_j & k = j\\   \eta_k & k\neq i,j\ .\end{cases}\ 
\end{equation}

\medskip

The jump process is then fully specified by giving the rate at which these pairwise stochastic exchanges occur. 
Associate to each pair $(i,j)$,  $i<j$, an exponential random variable  $T_{i,j}$ with parameter
\begin{equation}\label{jumprate}
\lambda_{i,j} =  N\ncht^{-1}(\eta_i + \eta_j)^{\gamma}\ ,
\end{equation}
where $0 \leq \gamma \leq 1$. In more detail, the $T_{i,j}$ are independent, and
$\mathbb{P}\{ T_{i,j} > t\} = e^{-t\lambda_{i,j}}$.
$T_{i,j}$ is the waiting time for  particles $i$ and $j$ to collide.
 The first collision occurs at time
\begin{equation}\label{alarm}
T = \min_{i<j}\{T_{i,j}\}\ .
\end{equation}

At the time $T$, the pair $(i,j)$ furnishing the minimum exchange energy in the prescribed way:   With $\alpha\in [0,1]$  selected uniformly at random,  the process jumps from
$ \bset$ to
$ \bset^*_{i,j,\alpha}$.  After each collision, the process starts afresh.
Let $\bset(t)$ denote the random state of the process at time $t$.

The object of our investigation is the spectral gap for the generator of the Markov semigroup associated to the stochastic exchange model described above.
For any continuous  function $f $ on $\STE$,
define the Kac walk generator $\LN$ by
$$\LN f ( \bset) = \frac1h \lim_{h\downarrow 0}{\mathbb E}\{ f ( \bset(h)) - f(\bset)\ |\  \bset(0) =  \bset\ \} \ ,$$
where the conditional expectation is taken with respect to the law of the process. 
The case $\gamma = 1/2$ models  the case of main physical interest, but the analysis of this case draws on the analysis for $\gamma =0$ and $\gamma =1$, and in the approach taken here, $\gamma =0$ and $\gamma =1$ are relatively simple, but for the rest, there is nothing to be gained by focusing on any one value of $\gamma\in (0,1)$.

One readily computes that
\begin{equation}
\LN f ( \bset) =  -{N}{\ncht}^{-1}\sum_{i<j} (\eta_i+\eta_j)^{\gamma}\
 \int_{0}^1 [f ( \bset) - f (\bset^*_{i,j,\alpha}))] \dd \alpha\ .
\end{equation}
Introducing the notation
${\displaystyle 
 [f ]^{(i,j)}( \bset)  :=    \int_{0}^1 f (\bset^*_{i,j,\alpha}) \dd \alpha\ ,
}$
the generator can be written as
\begin{equation}\label{lndef2}
\LN f ( \bset) =  -{N}{\ncht}^{-1}\sum_{i<j} (\eta_i+\eta_j)^{\gamma}\left[f ( \bset) - [f ]^{(i,j)}( \bset) \right]  \ .
\end{equation}

There is a somewhat more probabilistic way to write this: Let $\sigma_{N,E}$ denote the uniform probability measure on $\STE$. 
For $i < j$, let
$\scF_{i,j}$ be the $\sigma$-algebra $\scF_{i,j} := \sigma\{ \eta_k \:\ k\neq i,j\}$ .
Then 
$[f ]^{(i,j)} = {\mathbb E}\{ f | \scF_{i,j}\}$ where the conditional expectation is taken with respect to $\sigma_{N,E}$.

Let $\mathcal{H}_{N,E}$ denote the Hilbert space  $L^2(\STE,\sigma_{N,E})$.   One readily checks that $\LN$ is self-adjoint and bounded on $\mathcal{H}_{N,E}$. For $f,g\in \mathcal{H}_{N,E}$, we denote the inner product of $f$ and $g$ by $\langle f,g\rangle$ and the
norm of $f$ by $\|f\|_2$.
More generally, for a measurable function $f$ on $\STE$ and $1 \leq p \leq \infty$, $\|f\|_p$ denote the norm of $f$ in $L^p(\STE,\sigma_{N,E})$.

Let $\E$ be the Dirichlet form on  $\mathcal{H}_{N,E}$,\begin{equation}\label{dir3}
\E(f,f) = -\langle f,\LN f\rangle\ .
\end{equation}
Evidently, the uniform  distribution $\sigma_{N,E}$ is the unique equilibrium state for this process.

The object of our investigation, the {\em spectral gap}, is the quantity
\begin{equation}\label{gapdef}
\Delta_{\gamma,N,E} := \inf\{ \E(f,f) \ :\ f\in\mathcal{H}_{N,E}\ ,\ \|f\|_2 = 1\ ,\  \langle f,1\rangle =0\ \}\ .
\end{equation}

Because the jump rates defined in \eqref{jumprate} are homogeneous in $\bset\in \R_+^N$,  the dependence of $\Delta_{\gamma,N,E}$ on $E$ is a simple matter of scaling:

\begin{lm}\label{scale} For all $N>0$, $\gamma\in [0,1]$, and $E,E'>0$,
\begin{equation}\label{scaleB}
\Delta_{\gamma,N,E} = \left(\frac{E}{E'}\right)^\gamma \Delta_{\gamma,N,E'}\ .
\end{equation}
\end{lm}

\begin{proof}
Suppose $f$ is any measurable function on ${\mathcal S}_{N,E'}$. Define $Uf$ by
${\displaystyle Uf( \bset) = f\left( \frac{E'}{E} \bset\right)}$.
Then, $U$ is unitary from $\cH_{N,E'}$ to $\cH_{N,E}$ because $\sigma_{N,E'}$ is the push-forward of $\sigma_{N,E}$ under the map $\bset \mapsto 
\frac{E'}{E} \bset$.
For the same reason,
\begin{equation}\label{scaling}
N{\ncht}^{-1}\sum_{i<j}\int_{{\mathcal S}_{N,E'}}(\eta_i+\eta_j)^\gamma (f - [f]^{(i,j)})^2\dd \sigma_{N,E'} = 
\left(\frac{E'}{E}\right)^\gamma \E(Uf,Uf) \ .
\end{equation}
Then \eqref{scaleB} follows from \eqref{scaling} and the definition \eqref{gapdef}.
\end{proof}
This scaling relation will be important in the inductive estimation of $\Delta_{\gamma,N,1}$ that follows.
Note that in the case of uniform jump rates; that is, $\gamma = 0$, the value of $E$ is immaterial.  
On account of Lemma~\ref{scale}, we simplify our notation: We need only determine the spectral gap for unit energy per particle, and then we know it in general. 
Therefore, we define 
\begin{equation*}
\mathcal {S}_{N} = \mathcal{S}_{N,1},\quad\sigma_N:=\sigma_{N,1},\quad L_{\gamma, N} := L_{\gamma,N,1},\quad\mathcal{E}_{\gamma,N}:=\mathcal{E}_{\gamma,N,1},\quad  \Delta_{\gamma,N} := \Delta_{\gamma,N,1},\quad \mathcal{H}_N:=\mathcal{H}_{N,1} \ .
\end{equation*}
Our goal is then to determine $\Delta_{\gamma,N}$. Our approach will be inductive on $N$. For $\gamma>0$, it is difficult to compute, or even closely estimate, $\Delta_{\gamma,3}$. However, for $N=2$, it is a simple matter to exactly compute the spectral gap for all $\gamma$. 
When $N=2$,  there  is only one  pair $i< j$ and $\eta_i + \eta_j =2$ for all $\bset\in \mathcal{S}_{2}$.
Then from \eqref{lndef2}
\begin{equation}\label{lndef2B}
L_{\gamma, 2} f ( \bset) =  -2^{1+\gamma}\left[f ( \bset) - {\mathbb E}(f) \right]  \ .
\end{equation}
If follows that the eigenvalues of $L_{\gamma,2}$ are $0$ and $-2^{1+\gamma}$ and the $0$ eigenspace is spanned by the constant function $1$, and the
$-2^{1+\gamma}$ eigenspace is spanned by the functions orthogonal to $1$.  We record this as a lemma:

\begin{lm}\label{Delta2}
$\Delta_{\gamma,2} = 2^{\gamma+1}$ for all  $\gamma \geq 0$.  
\end{lm}

Our main result is:

\begin{thm}\label{CONJMAIN3} For all $\gamma\in [0,1]$ and all $N\geq 2$ there is a constant $C> 0$ depending only on $\gamma$ such that
\begin{equation}\label{SMNJK!2NM}
\Delta _{\gamma,N} \geq  C\ . 
\end{equation}
\end{thm}
The simplest  case is $\gamma=0$, which was already treated in \cite{C08,GF08}. The next simplest case is $\gamma =1$, and this was treated in \cite{S15}, which also deals with $\gamma>1$. However,  dealing with $0 < \gamma < 1$ is more involved and was open until now. As discussed in the next subsection, the case of greatest physical interest is $\gamma = 1/2$. 

In view of Lemma~\ref{Delta2},  even where it is  not explicitly stated, we tacitly assume $N\geq3$. 
Furthermore in the sequel ``a constant'' will be a number which may depend on $\gamma$ but not on $N$, and the value of the constant $C$ may change from line to line. 

\subsection{Physical context and previous results}

We close the introduction with a brief discussion of the physical context of this problem and the history of the progress made on it. Its origins lie  in  ground-breaking work of Gaspard and Gilbert \cite{GG08,GG09} on the long-standing problem of deriving Fourier's law of heat transfer in insulating  solid materials. In this work, as further developed by Grigo, Khanin, and Sz\'asz \cite{GKS12}, the stochastic exchange process described here arises for the specific value $\gamma = \tfrac12$. The physically relevant problem was to show that $\Delta_{\frac12,N}$ is bounded below by a strictly positive constant $C$, independent of $N$. 

The spectral gap problem for the stochastic exchange
model is closely related to the spectral gap problem for the
Kac model \cite{K56}. In the simplest version of the latter, the simplex is replaced by the 
sphere representing the velocities of $N$ particles with given total kinetic energy, and the dynamics consists of binary ``collisions'' that conserve the kinetic energy, or for velocities in $\R^d$, $d>2$, the kinetic energy and momentum.
In 1956, Kac conjectured a uniform lower bound on the gap for the simplest version of the model with one-dimensional velocities and
and uniform ``collision rates'', so that all kinematically possible collisions are equally likely. This conjecture was first proved by Janvresse \cite{J01}, and then the exact value of the gap was determined in \cite{CCL00,CCL03}, using an induction on the number of particles with controlled bounds on correlations. For further background on the Kac model, see \cite{CCL11}. 

Another method for bounding the spectral gap was developed by Caputo \cite{C08} that does not use induction on the number of particles. He showed that when the spectral gap for 3 particles is sufficiently large, then this directly implies a lower bound on the gap for all $N$; see
\cite[Theorem 2.1]{C08}. He applied this not only to the original Kac model, but what he calls the ``flat Kac model'', which is the stochastic exchange model for $\gamma =0$.  He notes that while his method is quite simple, it is somewhat less general. The Kac model spectral gap in $3$ dimensions with  energy and momentum  conserving collisions, but still with uniform collision rates, was computed in \cite{CGL}, and as Caputo notes, the gap for $N=3$ is too small for his method to apply.  About the same time, Geroux and Ferland treated a class of models including the stochastic exchange model for $\gamma =0$, using the methods of \cite{CCL00,CCL03}, and they too obtained the exact value of the spectral gap in this case. 

In the more physically meaningful versions of the Kac model and the stochastic exchange model, not all collisions occur with the same rates. Instead, the rates depend on the energy of the pair of particles, or other conserved quantities depending on the model, and the rates are generally degenerate. That is, the rate can be zero for some pairs. This leads to considerable mathematical difficulty, but for the Kac model with physically realistic ``hard-sphere collisions'' the problem of bounding the spectral gap from below uniformly in $N$ was solved in \cite{CCL14} in one dimension, and in \cite{CCL20} in three dimensions, using induction with controlled correlation bounds. 

Around the same time, Sasada was working on a class of stochastic exchange models, including the one discussed here. She adapted Caputo's method from \cite{C08} to non-uniform rates \cite{S15}; see also \cite{S13}.  She proved a uniform lower bound on the spectral gap for $\gamma \geq 1$, but her work for $\gamma<1$ depended on the calculation in
an appendix that was intended to show that also for $\gamma < 1$,  the gap for $N=3$ is sufficiently large to imply a uniform lower bound for all $N$. 
Unfortunately, it was discovered that these computations were not correct.

Here we apply the methods developed for the Kac model to the stochastic exchange model, and prove in Theorem~\ref{CONJMAIN3}  a uniform lower bound on the spectral gap in the specific case that is relevant to the physics discussed in \cite{GG08,GG09} and \cite{GKS12}. A broader class of models can be handled by the same method, and we briefly discuss this in the final section. 

An advantage of the inductive method that we develop here is that it allows one to make use of simplifications that emerge for large $N$. In particular, for large $N$, any pair of coordinate functions $(\eta_j,\eta_k)$ are approximately independent with respect to the uniform probability measure on $\STE$. If the correlations are already sufficiently small at $N=3$, then Caputo's method is very efficient. But verifying this, even for $N=3$ can be quite difficult. The method used here avoids such difficult calculations; one need only do careful calculations for large $N$, and then take advantage of simplifications due the the approximate independence. 
 In fact, in other work, Caputo \cite{Cap03} has made use of the inductive approach to treat uniform spectral gaps for diffusive Ginzburg-Landau model on the simplex, and he used Local Central Limit Theorem estimates, among other tools, to take advantage of the simplicity that emerges for large $N$.  

We now explain the induction with controlled correlations.

\section{The induction}

The inductive step consists of writing the Dirichlet form $\mathcal{E}_{\gamma,N}(f,f)$ in terms of an average over Dirichlet forms in which the $k^{{\rm th}}$ particle is ``frozen'' and not allowed to exchange energy with the others. What remains is the interaction between $N-1$ particles. Then the total energy possessed by these $N-1$ particles is $N -\eta_k$, and hence the energy per particle is ${\displaystyle \frac{N- \eta_k}{N-1}}$. Thus, the remaining $N-1$ particles carry out the same sort of walk on
$\mathcal{S}_{N-1,\frac{N-\eta_k}{N-1}}$.
These considerations lead to the definition of the 
 {\em conditional Dirichlet form} $\mathcal{E}_{\gamma,N}(f,f|\eta_k)$:
\begin{equation}\label{cdir}
{\mathcal E}_{\gamma,N} (f,f|\eta_k) := 
{(N-1)}{\nmcht}^{-1}
\sum_{i<j; i,j\neq k}\int_{\mathcal{S}_{N-1,\frac{N-\eta_k}{N-1}}}(\eta_i+\eta_j)^\gamma (f - [f]^{(i,j)})^2\dd \sigma_{N-1,\frac{N-\eta_k}{N-1}}\ ,
\end{equation}
noting that with the value of $\eta_k$ fixed, the remaining $N-1$ variables range over  a copy of $\mathcal{S}_{N-1,\frac{N-\eta_k}{N-1}}$ which is the sense in which the integral is to be taken.

Let $\dd \nu_N$ be the marginal distribution induced on $[0,N]$ by the map
$ \bset \mapsto \eta_k$ and the uniform probability measure $\dd\sigma_{N}$ on 
$\ST$, which is independent of $k$ by symmetry. Since $\sigma_N$ is a flat Dirichlet measure, it is well-known (and easy to see) that 
\begin{equation}\label{BETAMARG}
\dd \nu_N(\eta) = \frac{N-1}{N^{N-1}}(N- \eta)^{N-2}\dd \eta = 
\frac{N-1}{N}\left(1 - \frac{\eta}{N}\right)^{N-2}{\rm d}\eta\ .
\end{equation}
which is a rescaled $\beta$ distribution. Therefore, in the limit $N\to \infty$, ${\rm d}\nu_N$ converges to the exponential distribution 
$e^{-\eta}{\rm d}\eta$ on $[0,\infty)$. A simple computation shows that all moments of ${\rm d}\nu_N$ are bounded by those of the exponential distribution.

One easily checks that
\begin{equation}\label{rec}
{\mathcal E}_{\gamma,N}(f,f) = \frac{N}{N-1}
\left(\frac{1}{N}\sum_{k=1}^N  \int_0^N {\mathcal E}_{\gamma,N}(f,f|\eta_k) \dd \nu_N(\eta_k)\right)\ .
\end{equation}

Furthermore, 
\begin{equation}\label{rec2}
{\mathcal E}_{\gamma,N}(f,f|\eta_k)= {\mathcal E}_{\gamma,N-1,\frac{N-\eta_k}{N-1}}(g_{\eta_k},g_{\eta_k})
\end{equation}
where $g_{\eta_k}$ is the restriction of  $f$ to the 
slice of
$\mathcal{S}_{N-1,\frac{N-\eta_k}{N-1}}$ at constant  $\eta_k$. 
If $g_{\eta_k}$ were orthogonal to the constants in ${\cH_{N-1,\frac{N-\eta_k}{N-1}} }$, we could estimate the right hand side of \eqref{rec2}
in terms of $\Delta_{\gamma,N-1,\frac{N-\eta_k}{N-1}}$. By Lemma~\ref{scale}, 
${\Delta_{\gamma,N-1,\frac{N-\eta_k}{N-1}} = \left(\frac{N - \eta_k}{N-1}\right)^\gamma \Delta_{\gamma,N-1}}$.
Combining this with (\ref{rec2}) would yield (assuming $g_{\eta_k}$ is orthogonal to the constants)
\begin{equation}\label{rec3}
{\mathcal E}_{\gamma,N}(f,f|\eta_k)\geq   \left(\frac{N - \eta_k}{N-1}\right)^\gamma \Delta_{\gamma,N-1} \|g_{\eta_k}\|_2^2\ .
\end{equation}
By the Fubini-Tonelli theorem, $\|f\|_2^2 = \int_{[0,N]} \|g_{\eta_k}\|_2^2{\rm d}\nu_N(\eta_k)$ for any $k$. 
For $\gamma =0$, this together with \eqref{rec} would give us 
${\displaystyle {\mathcal E}_{0,N}(f,f) = \frac{N}{N-1}\Delta_{0, N-1}\|f\|_2^2}$,
which is too good to be true for all $f$, since in that case is would yield, by a simple induction starting from Lemma~\ref{Delta2}, that $\Delta_{0,N}\geq N$. This is demonstrably false. 

We must take into account that even if 
$f$ is orthogonal to the constants in 
$\mathcal{H}_N$, it need not be the case that $g_{\eta_k}$ 
is orthogonal to the  constants  on its slice.

\begin{defi}\label{PKOPDEF}
Define  the operator $P_k$ on $\cH_N$
to be the orthogonal projection onto the subspace of functions depending only on $\eta_k$. In probabilistic terms, $P_k$ is the operator of conditional expectation given $\eta_k$. 
\end{defi}

Then  the restriction of $f - P_kf$ to any slice determined by fixing a value for $\eta_k$ is orthogonal to the constants on the slice. Therefore,  
$${\mathcal E}_{\gamma,N}(f,f|\eta_k)  = {\mathcal E}_{\gamma,N}(f- P_kf,f-P_kf |\eta_k) \ .$$
Denoting  with $\|\cdot\|_{2,N-1,\frac{N-\eta_k}{N-1}}$ the norm on $\mathcal{H}_{N-1,\frac{N-\eta_k}{N-1}}$, by the definition of the spectral gap and scaling relation (\ref{scaleB}), 
\begin{equation*}
 \mathcal{E}_{\gamma,N}(f- P_kf,f-P_kf |\eta_k) \geq \left(\frac{N - \eta_k}{N-1}\right)^\gamma\Delta_{\gamma,N-1} \Vert f - P_k f\Vert^2_{2,N-1,\frac{N-\eta_k}{N-1}}\ .
\end{equation*}
 
Finally, we have the obvious identity
\begin{align*}
\Vert f - P_k f\Vert^2_{2,N-1,\frac{N-\eta_k}{N-1}} =
 \Vert f \Vert^2_{2,N-1,\frac{N-\eta_k}{N-1}} -
  \Vert  P_k f\Vert^2_{2,N-1,\frac{N-\eta_k}{N-1}}\ .
\end{align*}

Altogether, going back to (\ref{rec}), one has
\begin{equation}\label{rec2B}
 \mathcal{E}_{\gamma,N} (f,f) \geq \frac{N}{N-1}\Delta_{\gamma,N-1}
\left(\frac{1}{N}\sum_{k=1}^N  \int_{\mathcal{S}_N}\left(\frac{N - \eta_k}{N-1}\right)^\gamma
\left(f^2 - |P_kf|^2\right)  \dd \sigma_N\right)\ .
\end{equation}
To summarize  our conclusions, define the quadratic form $\mathcal{G}_{\gamma,N}(f,f) $ on $\cH_N$
as follows:
\begin{defi}\label{fformdef}
 \begin{equation}\label{find2}
\mathcal{G}_{\gamma,N}(f,f) := 
  \frac{1}{N} \sum_{k=1}^N\left[   \int_{\mathcal{S}_N} w^{(\gamma)}_N(\eta_k)(f - P_kf)^2\dd \sigma_N
 \right]\ 
\end{equation}
where
 \begin{equation}\label{find3}
 w^{(\gamma)}_N(\eta) := \left(\frac{N - \eta}{N-1}\right)^\gamma\ .
 \end{equation}
\end{defi}

We have proved:
 \begin{thm}\label{thm1} For all $\gamma\in [0,1]$ and  $f\in \cH_N$ with $\|f\|^2_2 =1$ and $\langle1,f\rangle = 0$,
 \begin{equation}\label{ind1}
 \mathcal{E}_{\gamma,N} (f,f) \geq \left[\frac{N}{N-1}\Delta_{\gamma,N-1}\right] \mathcal{G}_{\gamma,N}(f,f) \ .
\end{equation}
\end{thm}

We therefore define
\begin{defi}
\begin{equation}\label{gapdefinition}
\Gamma_{\gamma,N} := \inf\{ \mathcal{G}_{\gamma,N}(f,f) \ :\ f\in\mathcal{H}_N\ ,\ \|f\|_2 = 1\ ,\  \langle f,1\rangle =0\}\ .
\end{equation}
\end{defi}

Combining this definition with (\ref{gapdef}), Theorem~\ref{thm1} yields
\begin{equation}\label{red1}
\Delta_{\gamma,N} \geq   \Delta_{\gamma,N-1} \frac{N}{N-1} \Gamma_{\gamma,N}\ .
\end{equation}

For \eqref{red1} to provide a lower bound on $\Delta_{\gamma,N}$ that is uniform in $N$, we need a lower bound on $\Gamma_{\gamma, N}$ at least as strong as
\begin{equation}\label{NEED}
\Gamma_{\gamma, N} \geq 1 - \frac{1}{N} - A_N     
\end{equation}
where 
\begin{equation}\label{ANCOND}
0 \leq A_N < \frac{N-1}{N}\quad{\rm and}\quad \sum_{N=3}^\infty A_N < \infty\ .
\end{equation}
 This is because
${\displaystyle
\frac{N}{N-1}\left(1 - \frac{1}{N} - A_N\right)
=1 -  \frac{N}{N-1}A_N\ .
}$
Therefore,  with  $0 \leq A_N < \frac{N-1}{N}$ for all $N\geq 3$, and for any $N_0\geq 3$,
$$
\prod_{N=N_0}^\infty \frac{N}{N-1}\left(1 - \frac{1}{N} - A_N\right) >0
$$
if and only if the summability condition in \eqref{ANCOND} is satisfied. In particular, if one replaces
the second term on the right in \eqref{NEED} by $-\frac{C}{N}$ for any $C>1$, the infinite product would converge to zero, yielding a trivial bound.  Likewise, were it the case that $C<1$, one would obtain a spectral gap that diverges to infinity with $N$, and simple trial functional calculations show that this is impossible. Therefore, if the strategy is to work,  $\Gamma_{\gamma,N}$ must satisfy  $\Gamma_{\gamma,N} = 1 - \frac{1}{N} + o\left(\frac{1}{N}\right)$, and one must calculate the first two terms {\em exactly}
and then show that the remainder is summable. 

As we shall see, this is easily done for $\gamma=0$ and $\gamma=1$. However, the Dirichlet forms $\mathcal{G}_{\gamma,N}$ are not monotone in $\gamma$. Indeed, for all $\bset$ except $(1,\dots,1)$, there exists at least one index $k$ such that $\frac{N-\eta_k}{N-1} < 1$, 
and at least one index $k$ such that $\frac{N-\eta_k}{N-1}>1$. Consequently, the rates in the  Dirichlet forms $\mathcal{G}_{\gamma,N}(f,f)$ are not monotone in $\gamma$. In more probabilistic terms, it is not the case that any one of the corresponding processes simply runs faster than another. Therefore, there is no trivial monotonicity argument for passing estimates for $\gamma\in \{0,1\}$ to $\gamma\in (0,1)$. 

Most of our work goes into proving, in 
Section~\ref{LARGEN}, the following theorem:

\begin{thm}\label{CONJMAIN2} For all $\gamma\in [0,1]$ and all $N\geq 3$ there is a constant $C> 0$ depending only on $\gamma$ such that
\begin{equation}\label{SMNJK!2N}
\Gamma_{\gamma,N} \geq  \left(1 - \frac1N - \frac{C}{N^{3/2}}\right)\ . 
\end{equation}
\end{thm}

For small values of $N$, the lower bound provided by Theorem~\ref{CONJMAIN2} may be trivial; the right side of \eqref{SMNJK!2N} may be negative. Define $N_0$ to be the least value of $N$ such that the right side of \eqref{SMNJK!2N} is strictly positive. Then, provided we know  that $\Delta_{\gamma,N}>0$ for all $N \leq N_0$, we may begin the induction at $N_0$ to obtain the lower bound on $\Delta_{\gamma,N}$ claimed in our main result, Theorem~\ref{CONJMAIN3}.

It is relatively simple to prove that $\Delta_{\gamma,N}>0$ for all $N \leq N_0$. A uniform lower bound on $\Delta_{1,N}$ has already been proved by Sasada \cite[Theorem 4]{S15} using her refinement of Caputo's method \cite{C08}. As we show in 
Section~\ref{SMALLN}, a simple comparison argument leads from this to
a lower bound on $\Delta_{\gamma,N}$ that is not uniform, but does give us what we need to begin the induction at $N = N_0$, and to then apply Theorem~\ref{CONJMAIN2}. The next lemma is proved in
Section~\ref{SMALLN}.
\begin{lm}\label{CONJMAIN1} For all $\gamma\in [0,1]$ and all $N\geq 3$ there is a constant $C> 0$ depending only on $\gamma$ such that
\begin{equation}\label{SMNJK!2}
\Delta_{\gamma,N} \geq  C N^{\gamma -1}\ . 
\end{equation}
\end{lm}

Assuming for now the validity of Theorem~\ref{CONJMAIN2} and Lemma~\ref{CONJMAIN1}, we now prove our main result.

\begin{proof}[Proof of Theorem~\ref{CONJMAIN3}] Choose $N_0$ sufficiently large that $1 - N_0^{-1} - CN_0^{-3/2}> 0$ where $C$ is the constant in \eqref{SMNJK!2N}.
By Lemma~\ref{CONJMAIN1}, $\Delta_{\gamma,N_0} > 0$.  By \eqref{red1} and Theorem~\ref{CONJMAIN2}, for all $N> N_0$,
\begin{eqnarray*}
\Delta_{\gamma,N} &\geq& \Delta_{\gamma,N_0} \prod_{j=N_0+1}^N \frac{j}{j-1}\left(1 - \frac{1}{j}  - \frac{C}{j^{3/2}}\right)\\
&\geq& \Delta_{\gamma,N_0} \prod_{j=N_0+1}^\infty \left(1  - \frac{C}{j^{3/2}}\right)\ ,
\end{eqnarray*}
where $C$ has changed in the second line. The infinite product is strictly positive since $\sum_{j=N_0}^\infty j^{-3/2} < \infty$.  The lower bound on $\Delta_{\gamma,N}$ for $N < N_0$ is taken care of by Lemma~\ref{CONJMAIN1}.
\end{proof}

From this point on, our focus  is on proving Theorem \ref{CONJMAIN2} and Lemma~\ref{CONJMAIN1}. This comes down to a careful analysis of the 
Dirichlet form  $\mathcal{G}_{\gamma,N}$ and its spectral gap $\Gamma_{\gamma,N}$. Note that  the corresponding process is much simpler than the original stochastic exchange process. At each jump, one energy $\eta_k$  is fixed and all of the other particles ``jump to uniform'' on the slice given by the fixed value of $\eta_k$.  The rate for such a jump is degenerate, but mildly so: it vanishes only if the $k^{{\rm th}}$ particle has all of the energy. 

In proving Theorem~\ref{CONJMAIN2}, we make essential use of the fact that for large $N$, the random variables are almost independent. A number of precise expressions of this fact  that we shall need are given in Section~\ref{CHABNDS}. We now explain the importance of this asymptotic independence, and at the same time, why the cases $\gamma=0$ and $\gamma =1$ are so much simpler. 

For this purpose, it is useful to rewrite the Dirichlet form $\mathcal{G}_{\gamma,N}$ in another way. Define 
\begin{equation}\label{pgdef}
P^{(\gamma)} = \frac{1}{N}\sum_{k=1}^N \left(\frac{N - \eta_k}{N-1}\right)^\gamma P_k\ .
\end{equation}
For each $k$, both $P_k$ and the multiplication operator 
${\displaystyle \left(\frac{N - \eta_k}{N-1}\right)^\gamma}$ are commuting and self adjoint, so that
$P^{(\gamma)}$ itself is self-adjoint, and even non-negative. 
Since each $P_k$ is a projection, 
\begin{equation}\label{pg2}
\frac{1}{N}\sum_{k=1}^N \int_{{\mathcal S}_N}\left(\frac{N - \eta_k}{N-1}\right)^\gamma
|P_kf|^2 \dd \sigma_N
 = \langle f, P^{(\gamma)} f\rangle_{\cH_N}\ .
 \end{equation}
 
 Define the function $W^{(\gamma)}$ by
 \begin{equation}\label{wgdef}
W^{(\gamma)}(\boldsymbol{\eta}) := \frac{1}{N}\sum_{k=1}^N  w^{(\gamma)}_N(\eta_k) = \frac{1}{N}\sum_{k=1}^N \left(\frac{N - \eta_k}{N-1}\right)^\gamma \ .
\end{equation}

We may now  rewrite $\mathcal{G}_{\gamma,N}$ as
\begin{equation}\label{difform}
\mathcal{G}_{\gamma,N}(f,f)  = \int_{{\mathcal S}_N} W^{(\gamma)}_N f^2\dd \sigma_N
  - \langle f, P^{(\gamma)} f\rangle\ .
  \end{equation}
From  an upper bound of $P^{(\gamma)}$, and  a lower bound on $w^{(\gamma)}_N$, we can deduce  a lower bound on $\mathcal{G}_{\gamma,N}$.

Note that $W^{(0)} = W^{(1)} =1$.   Therefore, denoting with $\one$ the identity operator on $\mathcal{H}_N$,
\begin{equation}\label{SIMCAS}
\mathcal{G}_{0,N}(f,f) = \langle f,(\one  - P^{(0)})f\rangle \quad {\rm and}\quad  \mathcal{G}_{1,N}(f,f) = \langle f,( \one- P^{(1)})f\rangle_{\cH_N}\ .
  \end{equation}
Hence for these two cases, we only need information on the spectrum of $P^{(0)}$ and  $P^{(1)}$.  One can see from \eqref{wgdef} and the Law of Large Numbers  that if the random variables $\eta_1,\dots,\eta_N$ were exactly independent, then for large $N$, 
$W^{(\gamma)}$ would be very close to the constant ${\mathbb E}[W^{(\gamma)}]$. However this alone is not enough. The proof of Theorem~\ref{CONJMAIN2} is accomplished by showing that trial functions $f$ for which $\langle f,P^{(\gamma)}f\rangle_{\cH_N}$
is large, $f^2$ cannot concentrate too much mass in places where 
$W^{(\gamma)}$ is significantly below $1$.

\section{Chaoticity bounds}\label{CHABNDS}

The operator $P^{(0)}$ is an average over $N$ orthogonal projections. If  $\eta_1,\dots,\eta_N$ were independent random variables on 
${\mathcal S}_N$, these projections would commute, and the intersection of their ranges would be spanned by the constant function $1$. In this case, 
the spectrum of $P^{(0)}$ would be the set $\{1,\frac1N, 0\}$, with the constant function $1$ spanning the eigenspace with eigenvalue $1$, and the 
eigenspace for $\frac1N$ being spanned by functions of the form $f(\bset) = \varphi(\eta_k)$ for some $k$, with $\varphi$ such that ${\mathbb E}[\varphi(\eta_k)] =0$. 

Because of the constraint $\sum_{k=1}^N \eta_k = N$, $\eta_1,\dots,\eta_N$ are not independent, but for large $N$  we expect the dependence to be weak. 
In this section we prove several quantitative bounds on the dependence of these variables that are crucial for what follows.

\begin{defi}\label{KNHSDEF}
Let $\K_N$ denote the Hilbert space 
$L^2([0,N],\nu_N)$
where $\nu_N$ is the probability measure defined in \eqref{BETAMARG}.
\end{defi}

Define the map $T_N:[0,N]\times {\mathcal S}_{N-1} \to \mathcal{S}_N$ by
\begin{equation}\label{TNDEF}
T_N(\eta,{\boldsymbol \xi}) = \left(\frac{N-\eta}{N-1} \xi_1,\dots \frac{N-\eta}{N-1} \xi_{N-1}, \eta\right)\ .
\end{equation}
It is easy to check that the push-forward under $T_N$ of $\nu_N\otimes \sigma_{N-1}$ is $\sigma_N$.

\begin{defi}\label{KOPDEF}  The \emph{correlation operator} $K$ on $\K_N$ is defined by 
$$
K\varphi(\eta) = \mathbb{E}\{\varphi\left(\eta_{1}\right)|\eta_N = \eta\}\ .
$$
Defining $\pi_k:\mathcal{S}_N\to [0,N]$ by $\pi_k(\bset)= \eta_k$, we may relate $K$ to the operator $P_N$ specified in Definition~\ref{PKOPDEF}:
\begin{equation}\label{PKREL}
K\varphi(\eta_N)  = P_N(\varphi\circ\pi_{1})(\eta_N)\ .
\end{equation}
\end{defi} 

Note that by symmetry, in \eqref{PKREL} the indices $1$ and $N$ can be replaced by any pair of distinct indices, and in particular, the roles of $1$ and $N$ can be interchanged. Then by \eqref{PKREL},
 for $\varphi,\psi\in \K_N$,  
$$\langle \psi,K\varphi\rangle  = 
{\mathbb E}\left( (\psi\circ\pi_{N})
(\varphi\circ \pi_1)\right) = {\mathbb E}\left( (\psi\circ\pi_{1})
(\varphi\circ \pi_N)\right)  = \langle K\psi,\varphi\rangle\ .
$$
Thus, $K$ is self-adjoint, and we now determine its spectrum.

By \eqref{TNDEF},
\begin{equation}\label{Kform}
K\varphi(\eta) = \int_0^N \varphi\left( \frac{N-\eta}{N-1} x\right){\rm d}\nu_{N-1}(x) \ .
\end{equation}

From this explicit expression, we conclude that for any positive integer $m$, the space of polynomials of degree at most $m$ is invariant under $K$. 
To see this, let $p_n(x) := x^n$, $n$ a non-negative integer. Then,
using the map $T_N$ defined in \eqref{TNDEF}
\begin{equation}\label{Kpoly1}
Kp_n(\eta) = \left(\frac{N-\eta}{N-1}\right)^n  \int_0^{N-1} x^n{\rm d}\nu_{N-1}(x) = (N-\eta)^n \left((N-2)\int_0^1 y^n (1-y)^{N-3}{\rm d}y\right)\ .
\end{equation}
The integral on the right is a  $\beta$ function integral, and in any case,  integration by parts yields
\begin{equation}\label{BETAINT}
(N-2)\int_0^1 y^n (1-y)^{N-3}{\rm d}y = \frac{n!(N-2)!}{(n+N-2)!}.
\end{equation}
Thus the subspace ${\mathcal P}_n$ of $\K_N$ consisting of polynomials of degree $n$ or less is invariant under $K$.  Since $K$ is self adjoint, 
${\mathcal P}_n^\perp$ is also invariant, and then so is the one dimensional space  ${\mathcal P}_n \cap {\mathcal P}_{n-1}^\perp$. 
This is precisely the space spanned by the $n^\mathrm{th}$ polynomial in the sequence $\{\phi_n\}_{n\geq 0}$ of orthogonal polynomials obtained by applying the 
Gram-Schmidt procedure to the sequence $\{p_n\}_{n\geq 0}$ and using the inner product in $\K_N$. Hence for each $n$, $\phi_n$ is an eigenvector of $K$. Let $\kappa_n$ denote the associated eigenvalue. 

Then expanding the monomial $p_n$ in this basis,
${\displaystyle p_n = \sum_{j=0}^n \alpha_j\phi_j}$ where $\alpha_n \neq 0$, which we may write as $p_n = \alpha_n\phi_n + q$
where $q$ is a polynomial of degree at most $n-1$. Then 
$$\kappa_n\phi_n = K\phi_n = \frac{1}{\alpha_n}K\left( p_n - q\right) = (N-\eta)^n  \frac{n!(N-2)!}{(n+N-2)!} - Kq\ .$$
Therefore,  
$$
\kappa_n\phi_n - (-1)^n\eta^n \frac{n!(N-2)!}{(n+N-2)!}
$$
is a polynomial of degree at most  $n-1$. The left hand side of  \eqref{BETAINT} clearly decreases as $n$ increases, showing that $|\kappa_n|$ is a decreasing function of $n$. This proves:

\begin{thm}\label{KOPSPEC} Let $\{\phi_n\}_{n\geq 0}$ be the sequence of orthogonal polynomials obtained by applying the 
Gram-Schmidt procedure to the sequence $\{p_n\}_{n\geq 0}$. Then $\{\phi_n\}_{n\geq 0}$ is an orthonormal basis of $\K_N$ consisting of eigenvectors of $K$ wth $K\phi_n = \kappa_n\phi_n$ where 
\begin{equation}\label{KAPPn}
\kappa_n = (-1)^n\frac{n!(N-2)!}{(n+N-2)!}\ .
\end{equation}
For all $n$,  $|\kappa_{n+1}| \leq |\kappa_n|$.
\end{thm}

The first three eigenvalues are
\begin{equation}\label{Kpoly2}
\kappa_0 = 1\ ,\quad \kappa_1 = -\frac{1}{N-1}\quad{\rm and}\quad \kappa_2  = \frac{2}{N(N-1)}\ ,
\end{equation}
and the first two normalized eigenfunctions are
\begin{equation}\label{Kpoly2EF}
\phi_0 = 1\quad{\rm and}\quad \phi_1(\eta) = \sqrt{\frac{N+1}{N-1}}(\eta -1)\ .\
\end{equation}

The same method for determining the spectrum of the  analogous $K$ operator for the Kac model was used in \cite{CCL00,CCL03}, and has been applied to the 
$K$ operator for a class of models including the Stochastic Exchange model in \cite{GF08}. An analysis of the spectrum of the $K$ operator has been made by 
Caputo \cite{Cap03} in a setting in which the uniform probability measure on $\STE$ is replaced by a fairly general conditioned product measure, 
such as those discussed in the final section of this paper. While Caputo's result, \cite[Theorem~4.1]{Cap03}, covers a broad range of 
reference measures, the results are less precise than the specific bounds provided by 
Theorem~\ref{KOPSPEC}, and these bounds are needed in our application.

\begin{defi}\label{ANDEFN} Let $\cA_N$
denote the subspace of $\cH_N$ consisting of functions $f$ of the form 
\begin{equation}\label{ANDEF1}
f(\bset)= \sum_{j=1}^N \varphi_j(\eta_j)\ .
\end{equation}
with the further condition that ${\mathbb E}[f] =0$. 
\end{defi}
Note that ${\displaystyle \sum_{j=1}^N {\mathbb E}[\varphi_j(\eta_j)] =  {\mathbb E}[f]= 0}$. Therefore subtracting ${\mathbb E}[\varphi_j(\eta_j)]$ from 
$\varphi_j(\eta_j)$ for each $j$ does not change the sum in \eqref{ANDEF1}, and we may freely suppose that the functions $\varphi_j(\eta_j)$ in 
\eqref{ANDEF1} are such that ${\mathbb E}[\varphi_j(\eta_j)]= 0$.

Because $\sum_{j=1}^N (\eta_j -1) = 0$, for any $\alpha\in \R$,  we may also subtract $\alpha(\eta_j -1)$ from each $\varphi_j(\eta_j)$  
without changing the sum in \eqref{ANDEF1}.  We shall choose $\alpha$ to make ${\displaystyle \sum_{j=1}^N\|\varphi_j\|_2^2}$ as small as possible as $\alpha$ is varied.

Expand each $\varphi_j$ in the orthonormal basis $\{\phi_n\}_{n\geq 0}$
${\displaystyle
\varphi_j = \sum_{k=1}^\infty \alpha_{j,k}\phi_k }$
where the sum starts from $k=1$ since  $\alpha_{j,0} = \langle \phi_0,\varphi_j\rangle =  {\mathbb E}[\varphi_j(\eta_j)]= 0$.
Then
$$
\sum_{j=1}^N \|\varphi_j\|_2^2 = \sum_{j=1}^N\left( (\alpha_{j,1}-\alpha)^2 + \sum_{k=2}^\infty |\alpha_{j,k}|^2\right)\ .
$$
This is minimized by choosing $\alpha := \frac1N \sum_{j=1}^N \alpha_{j,1}$.   We make this choice, and note that consequently
\begin{equation}\label{ANDEF3}
\sum_{j=1}^N \alpha_{j,1} = 0\ .
\end{equation}
We have proved the following lemma:

\begin{lm}\label{CANREP} Every function $f\in \cA_N$ has a canonical representation  ${\displaystyle f(\bset)= \sum_{j=1}^N \varphi_j(\eta_j)}$ where, for $j=1,\dots,N$, $\varphi_j\colon \mathbb{R}\to\mathbb{R}$ is such that $\varphi(\eta_j)\in\mathcal{H}_N$
with
${\mathbb E}[\varphi_j(\eta_j)] =0$ and
\begin{equation}\label{CANREP1}
\sum_{j=1}^N {\mathbb E}[\varphi_j(\eta_j)(\eta_j -1)] = \sqrt{\frac{N-1}{N+1}}\sum_{j=1}^N {\mathbb E}[\varphi_j(\eta_j)\phi_1(\eta_j)]= 0\ .
\end{equation}
\end{lm} 

\begin{thm}\label{CHAOS} For $f\in \cA_N$, let  ${\displaystyle f(\bset)= \sum_{j=1}^N \varphi_j(\eta_j)}$ be its canonical representation as provided by Lemma~\ref{CANREP}. 
Then for all $N\geq 2$, 
\begin{equation}\label{ANDEF5}
\left(1 -  \frac{2}{N}\right )\sum_{j=1}^N\|\varphi_j\|_2^2 \leq \|f\|_2^2 \leq   \left(1 +  \frac{2}{N}\right)\sum_{j=1}^N\|\varphi_j\|_2^2 
\end{equation} 
\end{thm}

\begin{proof} We compute ${\displaystyle \|f\|_2^2 = \sum_{j=1}^N \|\varphi_j\|_2^2 + \sum_{j\neq k}{\mathbb E}[\varphi_j(\eta_j)\varphi_k(\eta_k)]}$. Define
$\widetilde{\varphi}_j := \varphi_j - \alpha_{j,1}\phi_1$.
Then since by  \eqref{ANDEF3}, 
$\sum_{k\colon k\neq j}\alpha_{k,1} = -\alpha_{j,1}$,
\begin{align*}
&\sum_{j\neq k}{\mathbb E}[\varphi_j(\eta_j)\varphi_k(\eta_k)]
=\sum_{j\neq k}\langle \varphi_j,K\varphi_k\rangle\\
&= -\sum_{j\neq k}\alpha_{j,1}\alpha_{k,1}\frac{1}{N-1}\langle \phi_1,\phi_1\rangle + 
  \sum_{j\neq k}\langle \widetilde{\varphi}_j,K\widetilde{\varphi}_k\rangle
  = 
  \frac{1}{N-1}\sum_{j=1}^N\alpha^2_{j,1} + 
  \sum_{j\neq k}\langle \widetilde{\varphi}_j,K\widetilde{\varphi}_k\rangle.
\end{align*}
Next, for $j\neq k$, 
$$|\langle \widetilde{\varphi}_j,K\widetilde{\varphi}_k\rangle| \leq \kappa_2\|\widetilde{\varphi}_j\|_2\|\widetilde{\varphi}_k\|_2 = \frac{2}{N(N-1)}\|\widetilde{\varphi}_j\|_2\|\widetilde{\varphi}_k\|_2 \leq \frac{1}{N(N-1)}(\|\widetilde{\varphi}_j\|_2^2+\|\widetilde{\varphi}_k\|_2^2)\ .$$
Therefore,
\begin{align*}
\sum_{j\neq k}|\langle \widetilde{\varphi}_j,K\widetilde{\varphi}_k\rangle|
\leq \frac{2}{N} \sum_{j=1}^N\left(\sum_{k=2}^\infty |\alpha_{j,k}|^2\right)
\leq \frac{2}{N} \sum_{j=1}^N\|\varphi_j\|_2^2\\ ,
\end{align*}
and \eqref{ANDEF5} follows.
\end{proof}

\begin{remark}\label{SUMREM}Let $\bigoplus^N \K_N$ denote the direct sum of $N$ copies of $\K_N$. Define  ${\mathcal B}_N$ to be the subspace of $\bigoplus ^N {\mathcal K}_{N}$ consisting of $(\varphi_1(\eta_1),\dots,\varphi_N(\eta_N))$ such that \eqref{CANREP1} is satisfied.  As a consequence of Theorem~\ref{CHAOS},
 the operator ${\mathcal T}: {\mathcal B}_N \to {\mathcal A}_N$ defined by 
 by
 $${\mathcal T}(\varphi_1(v_1),\dots,\varphi_N(v_N)) = \sum_{k=1}^N \varphi_k(v_k)$$
 is bounded and has a bounded inverse. 
 \end{remark}
 
 We now determine the spectrum of $P^{(0)}$.  An equivalent result has been obtained in \cite{GF08} by different method based on the original approach for the Kac model in \cite{CCL00,CCL03}. 
 
 \begin{lm}\label{MUNLEM}  For all $N\geq 3$, the spectrum of $P^{(0)}$ consists entirely of eigenvalues.
Let $\mu_N$ denote the second largest eigenvalue of $P^{(0)}$ after the eigenvalue $1$. Then  for $N\geq 3$, 
\begin{equation}\label{mu0val}
\mu_N = \frac1N  + \frac{2}{N^2}\ .
\end{equation}
\end{lm}

\begin{proof}The range of $P^{(0)}$ is ${\mathcal A}_N$, and it suffices to determine the spectrum of $P^{(0)}$ as an operator on ${\mathcal A}_N$. For 
${\displaystyle f(\boldsymbol{\eta}) = \sum_{j=1}^N \varphi_j(\eta_j)\in {\mathcal A}_N}$, we compute, using \eqref{PKREL} and symmetry in the coordinates,
$$P^{(0)} \left( \sum_{j=1}^N \varphi_j(\eta_j)\right) = \frac1N \sum_{j=1}^N \left(\varphi_j(\eta_j) + \sum_{k\neq j} K\varphi_k(\eta_j)\right)\ .$$
Let ${\mathcal T}: {\mathcal B}_N \to {\mathcal A}_N$ be defined as in Remark~\ref{SUMREM}. Then by  this computation, 
$ {\mathcal T}^{-1}P^{(0)} {\mathcal T}$
can be written as an $N\times N$ block matrix operator 
whose entries $M_{i,j}^{(0)}$ are given by 
$$M_{i,j}^{(0)} = \frac1N I  \quad{\rm if}\  i=j\quad{\rm and}\quad M_{i,j}^{(0)} = \frac1N K \quad{\rm if}\  i\neq j\ , $$
where $I$ is the identity operator on $\mathcal{K}_N$.

This  is  unitarily equivalent to the block matrix operator 
$$
\frac1N\left[\begin{array}{cccc} 
I + (N-1)K & 0    & \cdots & 0\\
0 & I-K  &\cdots & 0\\
\vdots & \vdots &  \ddots & \vdots \\
0 & 0   & 0 & I -K
\end{array}
\right]\ .
$$
By Theorem~\ref{KOPSPEC} $K$ is compact, and hence the spectrum of this operator consists entirely of eigenvalues. If $\lambda$ is any eigenvalue, it follows that either $\lambda$ is an eigenvalue of $I + (N-1)K$ or else $\lambda$ is an eigenvalue of $I - K$. Thus, the second largest eigenvalue of 
$P^{(0)}$, is 
either $\frac{1 + (N-1)\kappa_2}{N}$, where $\kappa_2$ is the second largest eigenvalue of $K$ namely $\frac{2}{N(N-1)}$, or else
$\frac{1-\kappa_1}{N}$ where $\kappa_1$ is the least eigenvalue of $K$, namely $-\frac{1}{N-1}$.  The two alternatives are
$$
\frac1N  + \frac{2}{N^2}
\quad{\rm and}\quad  
\frac1N + \frac{1}{N(N-1)}\ .
$$
The first is larger for all $N\geq 2$, and hence this gives the second largest eigenvalue.
\end{proof}

In Section~\ref{sec:sgh} we will need the following result,
 which also expresses a quantitative measure of the near-independence of the random variables $\eta_1,\dots,\eta_N$.
It is comparatively simple and can be obtained by explicit computation using \eqref{Kpoly1}, an obvious symmetry argument, and the elementary observation that for $\eta\in[0,N]$,  $\sigma_N(\cdot|\eta_N=\eta)=\sigma_{N-1,\frac{N-\eta}{N-1}}\otimes\delta_\eta$, where $\delta_\eta$ is the Dirac measure on $[0,N]$.

\begin{lm}\label{SUPBND1} For all $N\geq 3$, $i$, $j$ and $k$ distinct, and $\eta,\xi\geq0$ such that $\eta+\xi\leq N$, 
\begin{align*}
    {\mathbb E}\{ \eta_{i}^2 | \eta_j =\eta\}&=\frac{2(N-\eta)^2}{N(N-1)}\leq 3\ ,
    \qquad
    {\mathbb E}\{ \eta_{i}^2 | \eta_j =\eta,\eta_k =\xi\}=\frac{2(N-\eta-\xi)^2}{(N-1)(N-2)}\leq 9\ ,\\
  &{\mathbb E}\{ \eta_{i}^4 | \eta_j =\eta\}=\frac{24(N-\eta)^4}{(N+2)(N+1)N(N-1)}\leq 24\ .
\end{align*}
\end{lm}

\section{Bounds for small $N$}\label{SMALLN}

We begin by proving a uniform lower bound on $\Delta_{\gamma,N}$ for $\gamma =0$ and then $\gamma = 1$.  For $\gamma =0$, Caputo's method for bounding $\Delta_{0,N}$ in terms of $\Delta_{0,3}$ may be applied, and he has found \cite{C08} the sharp value $\Delta_{0,N} =\frac23 \frac{N+1}{N-1}$, and hence the sharp uniform bound $\Delta_{0,N} \geq \frac23$. (Equation (2.7) in \cite{C08} must be adjusted for the different normalization of the generator used there.) The same result was obtained at the same time by Giroux and Ferland \cite{GF08} using the inductive method. 
We present a short proof here as the simplest introduction to the use of the estimates obtained in the previous section. This result  covers the $\gamma=0$ case of Theorem~\ref{CONJMAIN3}. It is worth noting that the inductive method also yields the sharp value for the gap for $\gamma= 0$ in the stochastic exchange model, as it did for the Kac model \cite{CCL00}.

\begin{thm}\label{GAMZERGAP} The spectral gap $\Gamma_{0,N}$ is strictly positive for all $N\geq 3$, and satisfies
\begin{equation}\label{GAMZERGAP0}
\Gamma_{0,N} \geq \left(1 - \frac1N - \frac{2}{N^2}\right)\ .
\end{equation}

The spectral gap $\Delta_{0,N}$ is bounded below uniformly in $N\geq 3$. Specifically,
\begin{equation}
\Delta_{0,N} \geq \frac{2}{3} \frac{N+1}{N-1}\ .
\end{equation}
\end{thm}

\begin{proof}
First, for $\gamma = 0$, we have from 
\eqref{SIMCAS} and Lemma~\ref{MUNLEM} that for $f$ orthogonal to the constants
$$
 {\mathcal G}_{0,N} (f,f) =  \langle f,(\one - P^{(0)})f\rangle  \geq  \left(1 - \tfrac1N - \tfrac{2}{N^2}\right)\|f\|_2^2 \ .
$$
This shows that $\Gamma_{0,N} \geq 1 - \frac1N - \tfrac{2}{N^2} >0$
for all $N\geq 3$. 
Next, 
\begin{align*}
    \frac{N}{N-1}\Gamma_{0,N} \geq \left(1 + \frac{1}{N-1}\right)\left(1 - \frac1N - \frac{2}{N^2}\right)
= 1 - \frac{2}{N(N-1)} = \left(\frac{N+1}{N-1}\right)\left(\frac{N}{N-2}\right)^{-1}\ .
\end{align*}
Then by \eqref{red1},
\begin{align*}
\Delta_{0,N} \geq \Delta_{0,2}\prod_{j=3}^N \left(\frac{j+1}{j-1}\right)\left(\frac{j}{j-2}\right)^{-1} = \frac{2}{3} \frac{N+1}{N-1}\ ,      
\end{align*}
which is the exact lower bound found in \cite{C08,GF08}. This lower bound can be shown to be sharp by displaying the corresponding eigenfunction; see \cite{C08}.
\end{proof}

We next illustrate the inductive method by giving a short proof of Sasada's result \cite{S15} that $\Delta_{1,N}$ is bounded below, uniformly in $N$, by a strictly positive constant. 
This simple proof keeps our paper self-contained, and we can proceed almost exactly as in the proof of Theorem~\ref{GAMZERGAP};
we only need an upper bound on $\langle f, P^{(1)}f\rangle$ for $f$ orthogonal to the constants. 

\begin{lm}\label{projL} For all $N$, all $0 \leq \gamma \leq 1$, and all $f \in \cH_N$ orthogonal to the constants, 
\begin{equation}\label{proj2}
\langle f, P^{(\gamma)}f\rangle  \leq  \left(\frac{N}{N-1}\right)^\gamma\langle f, P^{(0)}f\rangle \ .
\end{equation}
\end{lm}

\begin{proof} Using the point-wise upper bound
${\displaystyle \left(\frac{N - \eta_k}{N-1}\right) \leq \frac{N}{N-1}}$, we have
\begin{multline}
\langle f,P^{(\gamma)} f\rangle = \frac{1}{N}\sum_{k=1}^N  \int_{{\mathcal S}_N}\left(\frac{N - \eta_k}{N-1}\right)^\gamma |P_kf|^2\dd \sigma_N
\leq\\
\frac{1}{N}\sum_{k=1}^N  \int_{{\mathcal S}_N}\left(\frac{N}{N-1}\right)^\gamma |P_kf|^2\dd \sigma_N
= \left(\frac{N}{N-1}\right)^\gamma \langle f,P^{(0)}f\rangle\ .
\end{multline}
\end{proof}

\begin{thm}\label{GAMONEGAP} The spectral gap $\Gamma_{1,N}$ is strictly positive for all $N\geq 3$, and satisfies
\begin{equation}
\Gamma_{1,N} \geq \left(1 - \frac{1}{N-1} - \frac{2}{N(N-1)}\right)\ .
\end{equation} 
The spectral gap $\Delta_{1,N}$ is uniformly bounded below in $N\geq 3$. Specifically,
\begin{equation}
\Delta_{1,N} \geq 4\prod_{j=3}^\infty\left(1-\frac{3}{(j-1)^2}\right)  >0\ .
\end{equation}
\end{thm}

\begin{proof}
First, for $\gamma = 1$, we have from 
\eqref{SIMCAS} and Lemma~\ref{projL} that for $f$ orthogonal to the constants
$$
 {\mathcal G}_{1,N} (f,f) =  \langle f,(\one - P^{(1)})f\rangle  \geq  \left\langle f,\left(1 - \frac{N}{N-1}P^{(0)}\right)f\right\rangle\ .
$$
By Lemma~\ref{MUNLEM},
$$
\left\langle f,\left(1 - \frac{N}{N-1}P^{(0)}\right)f\right\rangle \geq 1 - \frac{N}{N-1}\left(\frac1N + \frac{2}{N^2}\right) = 1 -\frac{1}{N-1} - \frac{2}{N(N-1)}\ .
$$
This shows that $\Gamma_{1,N} \geq 1 -\frac{1}{N-1} - \frac{2}{N(N-1)}$ which is strictly positive for all $N\geq 3$.  

Next, 
\begin{equation*}
\frac{N}{N-1}\Gamma_{1,N} \geq \left(1 + \frac{1}{N-1}\right)\left(1 - \frac{1}{N-1} - \frac{2}{N(N-1)}\right)
= 1 - \frac{3}{(N-1)^2}\ .
\end{equation*}
Since $\sum_{N=3}^\infty \tfrac{3}{(N-1)^2} < \infty$,  $\prod_{j=3}^\infty\left(1-\tfrac{3}{(j-1)^2}\right) > 0$, and by  Lemma~\ref{Delta2}, $\Delta_{1,2} =4$. 
\end{proof}

As a consequence of Theorem~\ref{GAMONEGAP}, we can give a simple proof that for all $\gamma\in (0,1)$ and all $\N \geq 3$, $\Delta_{\gamma,N}>0$
which is all we really need from Lemma~\ref{CONJMAIN1}.
Since 
$\frac{N-\eta}{N-1} \leq \frac{N}{N-1}$,
\begin{equation*}
w^{(\gamma)}_N(\eta) = \left(\frac{N - \eta}{N-1}\right)^\gamma
\geq \left(\frac{N-1}{N}\right)^{1-\gamma}w^{(1)}_N(\eta),
\end{equation*}
it follows from the definition \eqref{find2} that for all trial functions $f$, 
$$
\mathcal{G}_{\gamma,N}(f,f) \geq \left(\frac{N-1}{N}\right)^{1-\gamma}
\mathcal{G}_{1,N}(f,f) \geq \left(1 - (1-\gamma)\frac{1}{N-1}\right)
\mathcal{G}_{1,N}(f,f)\ 
$$
where the final inequality is a simple convexity estimate that gives up nothing of consequence. 

Combining this with \eqref{GAMZERGAP0} then yields that for all $\gamma\in (0,1)$ and $N\geq 3$.
\begin{eqnarray*}
\Gamma_{\gamma,N} &\geq& \left(1 - (1-\gamma)\frac{1}{N-1})\right)
\left(1 - \frac{1}{N-1} - \frac{2}{N(N-1)}\right)\\
&>& 1 - \frac{2-\gamma}{N-1} - \frac{(1+\gamma) }{(N-1)^2}  = 1 - \frac{(2-\gamma)N -3}{(N-1)^2}> 0\ .
\end{eqnarray*}
In fact, this argument almost gives a proof of Lemma~\ref{CONJMAIN1}. It shows that for any $\gamma' < \gamma$, and all sufficiently large $N$ (depending on the choice of $\gamma'$), $\Gamma_{\gamma,N} \geq \left(\frac{N-1}{N}\right)^{\gamma'}$.
Therefore, a simple induction using \eqref{red1} yields $\Delta_{\gamma,N} \geq CN^{\gamma'-1}$ for a constant $C$ depending only on $\gamma$ and $\gamma'$. 

The previous argument is simple but somewhat crude, and one may hope that a more direct application of the inductive method
using a sharp lower bound on $W^{(\gamma)}$ would yield a significantly better bound on $\Delta_{\gamma,N}$. The next lemma provides the lower bound on $W^{(\gamma)}$.

\begin{lm}\label{weight} For all $N$, all $0 < \gamma \leq 1$, and for all $ \bset \in {\mathcal S}_N$,
\begin{equation}\label{wlb1}
\left(\frac{N-1}{N}\right)^{1-\gamma} \ \leq \ W^{(\gamma)}( \bset)\  \leq\  1\ .
\end{equation}
\end{lm}

\begin{proof} Since
$ \frac{1}{N}\sum_{k=1}^N\eta_k =1$,
$ \frac{1}{N}\sum_{k=1}^N    \left(\frac{N - \eta_k}{N-1}\right) = 1$,
Jensen's inequality yields
$${ \frac{1}{N}\sum_{k=1}^N    \left(\frac{N - \eta_k}{N-1}\right)^\gamma \leq 
\left(\frac{1}{N}\sum_{k=1}^N\frac{N - \eta_k}{N-1}\right)^\gamma = 1\ .}$$

For the lower bound, since the function $ \bset \mapsto 
\frac{1}{N}\sum_{k=1}^N    \left(\frac{N - \eta_k}{N-1}\right)^\gamma$ is concave, it is minimized at an extreme point of ${\mathcal S}_N$. 
The extreme points $\bset$ are those for which $\eta_k = N$ for some $k$ with all other entries being zero. By symmetry,  the minimum occurs at $\bset = (N,0,\dots,0)$. 
\end{proof}

We are now ready to prove Lemma~\ref{CONJMAIN1}.

\begin{proof} [Proof of Lemma~\ref{CONJMAIN1}]
For all $f\in \cH_N$, $f$ orthogonal to the constants, by \eqref{difform}
\begin{equation}\label{difform2}
\mathcal{G}_{\gamma,N}(f,f)  = \int_{{\mathcal S}_N} W^{(\gamma)} f^2\dd \sigma_N
  - \langle f, P^{(\gamma)} f\rangle\ .
  \end{equation}
  Applying Lemma~\ref{projL}, Lemma~\ref{weight}
  and Lemma~\ref{MUNLEM} we get 
    \begin{eqnarray*}
 {\mathcal G}_{\gamma,N}(f,f)  &\geq& \left[\left(\frac{N-1}{N}\right)^{1-\gamma} -  \left(\frac{N}{N-1}\right)^{\gamma}\left(\frac{1}{N} +\frac{2}{N^2} \right)\right]\|f\|_2^2\\
 &=& \left(\frac{N}{N-1}\right)^{\gamma-2}\left[ \left(\frac{N}{N-1}\right) - \left(\frac{N}{N-1}\right)^2\left(\frac1N + \frac{2}{N^2}\right)\right]\|f\|_2^2\ .
  \end{eqnarray*}
  This shows that
    \begin{align*}
     & \Gamma_{\gamma,N} \geq   \left(\frac{N}{N-1}\right)^{\gamma-2}\left[ \left(\left(\frac{N}{N-1}\right) - \left(\frac{N}{N-1}\right)^2\frac1N\right) - 
  \left(\frac{N}{N-1}\right)^2 \frac{2}{N^2}\right]\\
  &= \left(\frac{N}{N-1}\right)^{\gamma-2}\left[ 1 - \frac{3}{(N-1)^2}\right]\ ,
  \end{align*}
which yields
\begin{align*}
 \Delta_{\gamma,N}  \geq \Delta_{\gamma,2} \prod_{j=3}^N  \left(\frac{j}{j-1}\right)^{\gamma-1}\left(1-\frac{3}{(j-1)^2}\right) \geq 
 \left(\frac{2}{N}\right)^{1-\gamma}\Delta_{\gamma,2} \prod_{j=3}^\infty\left(1-\frac{3}{(j-1)^2}\right) \ .
\end{align*}
This proves \eqref{SMNJK!2} because $\sum_{j=3}^\infty\frac{3}{(j-1)^2}<\infty$.
  \end{proof}
  
  Clearly, the weak point in the argument leading to Lemma~\ref{CONJMAIN1} is the use of the uniform lower bound on $W^{(\gamma)}$ provided by Lemma~\ref{weight}. The remedy to this is to show that  at least for large $N$, reasonable trial functions
  cannot concentrate much mass near the minimum of $W^{(\gamma)}$.

  This is facilitated by modifying the jump-rate function $w^{(\gamma)}_N$, replacing it with a  with a polynomial $m_N^{(\gamma)}$ such that $w^{(\gamma)}_N \geq m_N^{(\gamma)}$ everywhere. As we shall see, we get a sufficiently close fit using a quadratic polynomial.
  We then employ a trial function decomposition to show, using the precise calculations that we can make for the polynomial rates, that reasonable trial functions
  cannot concentrate much mass near the minimum of $W^{(\gamma)}$.  
The following simple lemma, repeated from \cite{CCL14}, is the starting point.

 \begin{lm}\label{comp1}
 For all $0 < \gamma < 1$ and all $x > -1$,
 \begin{equation}\label{com}
 (1+x)^ \gamma \geq 1 +  \gamma x - (1- \gamma)x^2\ .
 \end{equation}  
 \end{lm}
 
 \medskip
 
 \begin{proof}Let $\eta:[-1,\infty)\to \R$ be defined by
 $\eta(x) := (1+x)^ \gamma - ( 1 +  \gamma x - (1- \gamma)x^2)$.
 Note that
 $$ \eta''(x) = (1- \gamma)(2 -  \gamma(1+x)^{ \gamma-2})\ .$$
 Thus, $\eta''(x) =0$ has the single solution $x = x_*$ where
 $ x_*  := ( \gamma/2)^{1/(2- \gamma)} -1$.
 Note that $\eta$ is convex on $(x_*,\infty)$, and concave on $(-1, x_*]$, and also that
  $-1 < x_* < 0$. 
 
 Since $\eta$ is concave on $[-1,x_*]$, 
 $$\min\{ \eta(x)\ :\-1 \leq x \leq x_*\ \} = \min\{\eta(-1),\eta(x_*)\} = \eta(-1) = 0\ .$$
 Since  $\eta$ is convex  on $(x_*,\infty)$, and this interval contains a point, namely $0$,
 at which $\eta'$ vanishes, the minimum of $\eta$ over this interval is attained at $x=0$, and 
 thus $\eta$ is non-negative on $(x_*,\infty)$ as well as on  $[-1,x_*]$.
\end{proof}

Lemma \ref{comp1}  gives us the lower bound
\begin{equation}\label{WELOBND}
w^{(\gamma)}_N(\eta ) \geq m_N^{(\gamma)}(\eta) := 1 +\gamma \left(\frac{1 - \eta}{N-1}\right) 
-(1-\gamma)\left(\frac{1 - \eta}{N-1}\right)^2\ .
\end{equation}
Note that 
$$m_N^{(\gamma)}(\eta) = \left[1-(1-\gamma)\left(\frac{1 - \eta}{N-1}\right) \right]\left[ \left(\frac{1 - \eta}{N-1}\right)  +1\right] > 0\ .$$
Then
  \begin{equation}\label{wabnd1B}
  \mathcal{G}_{\gamma,N}(f,f)\geq  \frac{1}{N} \sum_{k=1}^N\left[   \int_{\ST} m_N^{(\gamma)}(\eta_k)[f - P_kf]^2\dd \sigma_N
 \right]  := \widetilde{\mathcal{G}}_{\gamma,N}(f,f)\ .
 \end{equation}
 The following notation will be convenient. 
\begin{align}
 &\widetilde{ W}^{({\gamma})}(\boldsymbol{\eta}) := \frac1N\sum_{k=1}^N m_N^{(\gamma)}(\eta_k) =  1 - \frac{1-\gamma}{(N-1)^2}\frac1N  \sum_{k=1}^N (\eta_k^2 -1) \geq 1 -  \frac{1-\gamma}{N-1}\ ,\label{wabnd1}\\
 &\widetilde{ P}^{(\gamma)} := \frac1N\sum_{k=1}^N m_N^{(\gamma)}(\eta_k) P_k \label{wabnd2}\ ,\\ 
&\Dto (f,f) := \int_{\ST}  \widetilde{ W}^{({\gamma})} f^2 {\rm d}\sigma_N - \langle f, \widetilde{P}^{(\gamma)} f\rangle\label{Dtodef}\ .
\end{align}

\begin{remark}\label{GTREM}
Because of \eqref{wabnd1B}, any spectral gap bound for the Dirichlet form $\Dto$ serves as a spectral gap bound for the 
Dirichlet form ${\mathcal G}_{\gamma,N}$.
\end{remark}

\section{The trial function decomposition}\label{sec:sgh}

We begin by specifying a trial function decomposition that we shall use. Let ${\mathcal A}_N$ be the subspace of $\cH_N$ specified in Definition~\ref{ANDEFN}. 
For any $f\in \cH_N$ orthogonal to the constants, define $p$ and $h$ to be the orthogonal projections of $f$ onto ${\mathcal A}_N$ and ${\mathcal A}_N^\perp$ respectively. 
Then since $1\in {\mathcal A}_N^\perp$, $p$ is orthogonal to the constants, and then $h =f- p$ is orthogonal to the constants. The following lemma motivates this decomposition. Throughout this section we work with  $\widetilde{ W}^{({\gamma})}(\boldsymbol{\eta}) $ and $ \widetilde{ P}^{(\gamma)}$ as defined in \eqref{wabnd1} and \eqref{wabnd2}.

\begin{lm}\label{nullalpha}  For all $\gamma \in [0,1]$ the null space of $\widetilde{P}^{(\gamma)}$ is the intersection of the null spaces of $P_k$ for $k=1,\dots N$.
For all $\gamma \in [0,1]$, the closure of the range of $\widetilde{P}^{(\gamma)}$ is the subspace 
${\mathcal A}_N$  of $\cH_N$
given in Definition~\ref{ANDEFN}. 
\end{lm}

\begin{proof} Since  $\widetilde{ P}^{(\gamma)}\geq 0$,  $f$ belongs to the null space of $\widetilde{ P}^{(\gamma)}$ if and only if 
$\langle f, \widetilde{ P}^{(\gamma)} f \rangle = 0$.
Because $m_N^{(\gamma)}(\eta_k)    > 0$ almost everywhere 
  $${ 0 =  \langle f,  \widetilde{ P}^{(\gamma)}f\rangle
  = \frac{1}{N}\sum_{k=1}^N\int_{{\mathcal S}_N}m_N^{(\gamma)}(\eta_k)    |P_k f|^2\dd \sigma_N}\ ,$$
it must be the case that $|P_k f|^2$ vanishes identically.  Thus, if $f$ is in the null space of $\widetilde{ P}^{(\gamma)}$, 
then $P_k f = 0$ for each $k$, and $f$ is in the null space of 
$P^{(0)}$. Conversely if $f$ is in the null space of $P^{(0)}$, then $P_k f = 0$ for each $k$, and then clearly $\widetilde{ P}^{(\gamma)}f = 0$. 

Since $\widetilde{ P}^{(\gamma)}\geq 0$, the closure of its range is the orthogonal complement of its null space. Thus the null space does not depend on $\gamma$, neither does the range. Evidently, ${\mathcal A}_N$ is the closure of the range of $P^{(0)}$. 
\end{proof}

Let $f = p+h$ be the trial function decomposition described at the beginning of this section. By Lemma~\ref{nullalpha},  
$h$ is  the component of $f$ in the null space of  $ \widetilde{ P}^{(\gamma)}$ for each $\gamma\in [0,1]$, and hence 
\begin{equation}\label{red2Y}
\langle f, \widetilde{ P}^{(\gamma)} f \rangle  = \langle p, \widetilde{ P}^{(\gamma)}p \rangle\ .
\end{equation}
Using this in \eqref{Dtodef}  yields

\begin{equation}\label{red2X} \widetilde{ {\mathcal G}}_{\gamma,N}(f,f) = 
  \int_{\ST} \widetilde{ W}^{({\gamma})} f^2\dd \sigma_N
  - \langle p, \widetilde{ P}^{({\gamma})} p\rangle \ .
 \end{equation}
  
  Since $p\in {\mathcal A}_N$,  it has a canonical decomposition of Lemma~\ref{CANREP}
\begin{equation}\label{struc}
p(\bset) = \sum_{j=1}^N\rho_j(\eta_j)\ 
\end{equation}
where  
\begin{equation}\label{struc2}
\sum_{j=1}^N{\mathbb E}[\rho_j(\eta_j)(1-\eta_j)] =0\quad{\rm and\ for\ each} \  j, \quad {\mathbb E}[\rho_j(\eta_j)] =0\ .
\end{equation}

We make a further decomposition of $\rho_j(\eta_j)$ as follows: 

\begin{defi}\label{decompdef}  Let  $p$ be  a function given by a sum of the form \eqref{struc} subject to the constraints \eqref{struc2}.
Expand each $\rho_j$ in the orthonormal basis $\{\phi_n\}_{n\geq 0}$ of eigenfunctions of $K$.  
Define
\begin{equation}\label{sDEF}
s := \sum_{j=1}^N\langle \phi_1,\rho_j\rangle \phi_1(\eta_j) \quad{\rm and\ then}\quad g := p -s=\sum_{j=1}^N\sum_{n=2}^\infty\langle \phi_n,\rho_j\rangle \phi_n(\eta_j)\ .
\end{equation}
To simplify notation we define
\begin{equation}\label{sDEFB}
\alpha_j := \langle \phi_1,\rho_j\rangle \ ,\quad  \psi_j(\eta_j) := \alpha_j\phi_1(\eta_j) \quad{\rm and}\quad \varphi_j(\eta_j) := 
\sum_{n=2}^\infty\langle \phi_n,\rho_j\rangle \phi_n(\eta_j)\ .
\end{equation}

The {\em trial function decomposition} of any $f\in \cH_N$ that is orthogonal to the constants is given by
\begin{equation}\label{pdecomp}
f
=p+h
 =  s+g + h
=\sum_{j=1}^N\alpha_j\phi_1(\eta_j)+\sum_{j=1}^N\sum_{n=2}^\infty\langle \phi_n,\rho_j\rangle \phi_n(\eta_j)+h
=\sum_{j=1}^N\psi_j+\sum_{j=1}^N\varphi_j+h
\end{equation}
where $p$ is the component of $f$ in the closure of the range of $P^{(\gamma)}$, $h$ is the component of $f$ in the null space of $P^{(0)}$, and 
$p =s+g$ is the decomposition of $p$ defined in \eqref{sDEF}.
\end{defi}

The next lemma gives a crucial upper bound on $\langle g, P^{(\gamma)} g\rangle$. In proving it, we make use of Theorem~\ref{KOPSPEC}
which has the consequence that for each $j$, 
\begin{align}\label{eq:key}
    \|K\varphi_j\|_2 \leq \kappa_2\|\varphi_j\|_2 = \frac{2}{N(N+1)}\|\varphi_j\|_2\ .
\end{align}
Had we not split off the $\phi_1$ component of $\rho_j$, we would have had to use $\kappa_1$ in place of $\kappa_2$ in this estimate, which would degrade it by a factor of $N$.  This is the reason for splitting $p =s+g$. Since $s$ is affine, we can treat it separately by other means, as we shall see. 

At this point is worth remarking that if the trial function $f$ is symmetric under coordinate permutations, then $s=0$ because $\sum_{j=1}^N \alpha_j =0$
and $\alpha_j$ is independent of $j$.  This was taken advantage of   in \cite{CCL14} to ``skip over''  all of the estimates on $s$, and produce a sharper bound on the spectral gap for symmetric $f$, the case relevant in kinetic theory. Here we do not keep track of constants, but we do estimate the contribution form $s$ so that there is no symmetry restriction in our spectral gap bound. 

\begin{lm}\label{gPUB}  Let $f\in \cH_N$ be orthogonal to the constants and let $f = s+g + h$ be its trial function decomposition \eqref{pdecomp},
so that $g$ has the form
$g = \sum_{j=1}^N\varphi_j(\eta_j)$. Then there is a constant $C$ independent of $N$ and $f$  such that 
\begin{equation}\label{gPUB1}
\langle g, \widetilde{  P}^{(\gamma)} g\rangle \leq \frac1N \sum_{k=1}^N \int_{\ST} m_N^{(\gamma)}(\eta_k) |\varphi_k(\eta_k)|^2{\rm d}\sigma_N  + \frac{C}{N^2}\|g\|_2^2\ .
\end{equation}
\end{lm}

\begin{proof} Since $P_k g = \varphi_k(\eta_k) +\sum_{j\colon j\neq k}K\varphi_j(\eta_k)$,
\begin{eqnarray}
\langle g, P^{(\gamma)} g\rangle &=& \frac1N \sum_{k=1}^N \int_{\ST} m_N^{(\gamma)}(\eta_k) \left(\varphi_k(\eta_k) +\sum_{j\colon j\neq k}K\varphi_j(\eta_k)\right)^2{\rm d}\sigma_N\nonumber\\
&=& \frac1N \sum_{k=1}^N \int_{\ST} m_N^{(\gamma)}(\eta_k) |\varphi_k(\eta_k)|^2{\rm d}\sigma_N\nonumber\\
&+&  \frac2N \sum_{k=1}^N\sum_{j\colon j\neq k} \int_{\ST} m_N^{(\gamma)}(\eta_k)\varphi_k(\eta_k)K\varphi_j(\eta_k){\rm d}\sigma_N\label{gPUB3}\\
&+& \frac1N \sum_{k=1}^N\sum_{j\colon j\neq k}\sum_{\ell\colon\ell\neq k} 
 \int_{\ST} m_N^{(\gamma)}(\eta_k)K\varphi_j(\eta_k)K\varphi_\ell(\eta_k){\rm d}\sigma_N\label{gPUB4}
\end{eqnarray}
 Using the pointwise  bound $m_N^{(\gamma)}(\eta_k) \leq w_N^{(\gamma)}(\eta_k) \leq \left(\frac{N}{N-1}\right)^\gamma$ and \eqref{eq:key}, for $j\neq k$ ,
$$
\left| \int_{\ST} \left(\frac{N-\eta_k}{N-1}\right)^\gamma\varphi_k(\eta_k)K\varphi_j(\eta_k){\rm d}\sigma_N\right| \leq 
\left(\frac{N}{N-1}\right)^\gamma \kappa_2\|\varphi_j\|_2\|\varphi_k\|_2\ .
$$
Therefore, the expression in line \eqref{gPUB3} is bounded above by $ \left(\frac{N}{N-1}\right)^{\gamma -1}\kappa_2 \sum_{k=1}^N\|\varphi_k\|_2^2$, and then by \eqref{Kpoly2} and Theorem~\ref{CHAOS}, this is bounded above by $\frac{C}{N^2}\|g\|_2^2$.  In the same way, one shows that the sum of the integrals 
in \eqref{gPUB4} is bounded above by $\frac{C}{N^3}\|g\|_2^2$.
\end{proof}

The components of $s$, which lie in the eigenspace of $K$ with eigenvalue $\kappa_1=-\frac{1}{N-1}$, are more correlated with one another than the components of $g$, but $s$ has other nice properties given in the next two lemmas.

\begin{lm}\label{SEIG} Let $f\in \cH_N$ be orthogonal to the constant and let $f = s+g + h$ be its trial function decomposition \eqref{pdecomp}, so that $s$ has the form 
$s = \sum_{j=1}^N\alpha_j\phi_1(\eta_j)=\sum_{j=1}^N\psi_j(\eta_j)$ with $\sum_{j=1}^N \alpha_j = 0$. 
 Then for each $k$
\begin{equation}\label{SEIG1}
P_k s= \frac{N}{N-1}\alpha_k\phi_1(\eta_k)=\frac{N}{N-1}\psi_k(\eta_k)\ .
\end{equation}
\end{lm}

\begin{proof} Recall from Definition~\ref{KOPDEF}
that for $1\leq k \leq N$,  $\pi_k:\mathcal {S}_{N} \to \R$ is defined by
$\pi_k(\boldsymbol{\eta}) = \eta_k$. For $j\neq k$, $P_k (\phi_1\circ \pi_j)  = -\frac{1}{N-1}\phi_1\circ \pi_k$. Therefore,
$$
P_k s   =\alpha_k\phi_1(\eta_k)  -  \sum_{k\neq j} \frac{1}{N-1}\alpha_j\phi_1(\eta_k) = \frac{N}{N-1}\alpha_k \phi_1(\eta_k)\ 
$$
where in the last equality we have used $\sum_{j=1}^N \alpha_j = 0$. 
\end{proof}

As a consequence of Lemma~\ref{SEIG}, we obtain the analog of \eqref{gPUB1}, as an identity instead of as an inequality:

\begin{equation}
\label{sPUB1}
    \begin{split}
\langle s, \widetilde{P}^{(\gamma)}s\rangle &=\frac1N \sum_{k=1}^N \int_{\ST}  m_N^{(\gamma)}(\eta_k) | P_ks|^2\, {\rm d}\sigma_N\\
    &= \left(\frac{N}{N-1}\right)^2   \frac1N \sum_{k=1}^N \int_{\ST}  m_N^{(\gamma)}(\eta_k) |\alpha_k\phi_1(\eta_k)|^2 {\rm d}\sigma_N.
    \end{split}
\end{equation}

\begin{lm}\label{sL4B}  Let $f\in \cH_N$ be orthogonal to the constant and let $f = s+g + h$ be its trial function decomposition \eqref{pdecomp}, so that $s$ has the form 
$s = \sum_{j=1}^N\alpha_j\phi_1(\eta_j)$ with $\sum_{j=1}^N \alpha_j = 0$.  Then there is a constant $C$ independent of $N$ and $f$ such that
\begin{equation}\label{sL4B1}
\|s\|_4 \leq CN^{1/2}\|s\|_2\ .
\end{equation}
\end{lm}

\begin{proof} By the Cauchy-Schwarz inequality \eqref{Kpoly2EF}, and the fact that $\sigma_N$-almost surely $\sum_{j=1}^N\eta_j=N$,
\begin{align*}
|s| \leq \sqrt{2}\left(\sum_{j=1}^N\alpha_j^2\right)^{1/2}\left(\sum_{j=1}^N(\eta _j-1)^2\right)^{1/2} \leq 
\sqrt{2}\left(\sum_{j=1}^N\alpha_j^2\right)^{1/2}\left(\sum_{j=1}^N \eta _j^2\right)^{1/2} \ .
\end{align*}
Then 
${\displaystyle
\|s\|_4^4 \leq 4\left(\sum_{j=1}^N\alpha_j^2\right)^{2} \int_0^N \left(\sum_{j=1}^N \eta _j^2\right)^{2} {\rm d}\nu_N
}$. 
Using first the triangle inequality in $\mathcal{H}_N$  and then Lemma~\ref{SUPBND1}, 
\begin{equation}\label{TRINQ}
\left(\int_0^N \left(\sum_{j=1}^N \eta _j^2\right)^{2} {\rm d}\nu_N\right)^{1/2} \leq 
\sum_{j=1}^N \left(\int_0^N  \eta _j^4 {\rm d}\nu_N\right)^{1/2}  \leq   \sqrt{24} N\ .
\end{equation}
Then $\|s\|_4^2 \leq  2\sqrt{24} N  \sum_{j=1}^N\alpha_j^2$, and Theorem~\ref{CHAOS}  bounds
$\sum_{j=1}^N\alpha_j^2$ by a multiple of $\|s\|_2^2$.
\end{proof}

The remaining lemmas in this section will be used in the next section to take advantage of the polynomial nature of the modified rates. On a first reading, it may help to skip ahead to the next section to see how they are used, and to then return.

\begin{lm}\label{gv4lem} Let $f\in \cH_N$ be orthogonal to the constants and let $f = s+g + h$ be its trial function decomposition \eqref{pdecomp},
so that $g$ has the form
$g = \sum_{j=1}^N\varphi_j(\eta_j)$ and $s$ has the form 
$s = \sum_{j=1}^N\alpha_j\phi_1(\eta_j)$ with $\sum_{j=1}^N \alpha_j = 0$.
Then there is a constant $C$ independent of $N$ and $f$ such that 
\begin{equation}\label{gv4}
 \sum_{k=1}^N\int_{\ST}   \eta_k^2 g^2{\rm d}\sigma_N  \leq  \sum_{k=1}^N \int_{\ST}  \varphi_k(\eta_k)^2\eta_k^2 {\rm d}\sigma_N + {C}{N}\|g\|_2^2\ ,
 \end{equation}
 and
 \begin{equation}\label{gv4B}
\sum_{k=1}^N \int_{\ST}    \eta_k^2 s^2{\rm d}\sigma_N \leq  \sum_{k=1}^N \int_{\ST}  \psi_k(\eta_k)^2\eta_k^2 {\rm d}\sigma_N + {C}{N}\|s\|_2^2\ .
 \end{equation}
\end{lm}
\begin{proof} We first consider \eqref{gv4}. 
\begin{eqnarray}
\sum_{k=1}^N\int_{\ST}  \eta_k^2 g^2 {\rm d}\sigma_N &=& \sum_{i=1}^N\sum_{j=1}^N \sum_{k=1}^N\int_{\ST}  \varphi_i(\eta_i)\varphi_j(\eta_j)\eta_k^2 {\rm d}\sigma_N\nonumber\\
&=&  \sum_{k=1}^N \int_{\ST}  \varphi_k(\eta_k)^2\eta_k^2 {\rm d}\sigma_N\label{chi1}\\
&+& 2\sum_{j\neq k} \int_{\ST}  \varphi_j(\eta_j)\varphi_k(\eta_k)\eta_k^2 {\rm d}\sigma_N\label{chi2}\\
&+& \sum_{j\neq k} \int_{\ST}  \varphi_j(\eta_j)^2\eta_k^2 {\rm d}\sigma_N\label{chi3}\\
&+& \sum_{|\{i,j, k\}|=3} \int_{\ST} \varphi_i(\eta_i)\varphi_j(\eta_j)\eta_k^2 {\rm d}\sigma_N,\label{chi4}
\end{eqnarray}
where the sums in \eqref{chi2} and \eqref{chi3} are on  $j$ and $k$ such that $j\neq k$ while \eqref{chi4} is on $i$, $j$ and $k$ such that the set $\{i,j, k\}$ has cardinality 3.

For $j\neq k$, 
\begin{eqnarray*}
\left|\int_{\ST}  \varphi_j(\eta_j)\varphi_k(\eta_k)\eta_k^2 {\rm d}\sigma_N\right| =  \left|\int_{\ST}  K\varphi_j(\eta)\varphi_k(\eta )\eta^2 {\rm d}\nu_N\right|
&\leq& \kappa_2 N^2 \|\varphi_j\|_2\|\varphi_k \|_2\\ &\leq& \frac{N}{N+1}(\|\varphi_j\|_2^2 +\|\varphi_k\|_2^2)\ .
\end{eqnarray*}
Therefore the sum in \eqref{chi2} is bounded above by $4N\sum_{j=1}^N\|\varphi_j\|_2^2$ and then Theorem~\ref{CHAOS} yields an upper bound of the form
$CN\|g\|_2^2$. 

Next, by
Lemma~\ref{SUPBND1} 
$$\int_{\ST}  \varphi_j(\eta_j)^2\eta_k^2 {\rm d}\sigma_N = \int_{\ST}  \varphi_j(\eta_j)^2{\mathbb E}\{\eta_k^2 \ |\ \eta_j\} {\rm d}\sigma_N \leq 3\|\varphi_j\|_2^2\ .$$
Then again using Theorem~\ref{CHAOS}, the sum in \eqref{chi2} is bounded above by $CN\|g\|_2^2$\ .

Finally, for $i$, $j$ and $k$ are distinct, 
\begin{equation*}
 \int_{\ST} \varphi_i(\eta_i)\varphi_j(\eta_j)\eta_k^2 {\rm d}\sigma_N = 
 \int_{\ST} \varphi_i(\eta_i)\varphi_j(\eta_j){\mathbb E}\{\eta_k^2|\eta_i,\eta_j\} {\rm d}\sigma_N\ .
 \end{equation*}
By  Lemma~\ref{SUPBND1} and simple algebra,
$$
{\mathbb E}\{\eta_k^2|\eta_i,\eta_j\}  = \frac{2N^2}{(N-2)(N-1)} +   \frac{4N}{(N-2)(N-1)}(\eta_i+\eta_j) + 
 \frac{2}{(N-2)(N-1)} (\eta_i^2+2\eta_i\eta_j +\eta_j^2)\ .
$$
Therefore, to estimate ${\int_{\ST} \varphi_i(\eta_i)\varphi_j(\eta_j){\mathbb E}\{\eta_k^2|\eta_i,\eta_j\} {\rm d}\sigma_N}$, we need estimates on 
\begin{equation}\label{DECLEMC1}
\begin{split}
 \int_{\ST} \varphi_i(\eta_i)&\varphi_j(\eta_j){\rm d}\sigma_N\ ,\qquad\qquad 
  \int_{\ST} \varphi_i(\eta_i)\varphi_j(\eta_j)(\eta_i + \eta_j){\rm d}\sigma_N, \\ 
  &\int_{\ST} \varphi_i(\eta_i)\varphi_j(\eta_j) (\eta_i^2+2\eta_i\eta_j +\eta_j^2){\rm d}\sigma_N\ . 
\end{split}
\end{equation}

First,  
\begin{equation}\label{chi2MOD}
\left|\int_{\ST} \varphi_i(\eta_i)\varphi_j(\eta_j){\rm d}\sigma_N\right| = |\langle \varphi_i,K\varphi_j\rangle| \leq 
\kappa_2\|\varphi_i\|_2\|\varphi_j\|_2 \leq \frac{1}{N(N+1)}(\|\varphi_i\|_2^2 + \|\varphi_j\|_2^2 )\ .
\end{equation}
Then using Theorem~\ref{CHAOS},  the contribution of the first integral in \eqref{DECLEMC1} to the sum of integrals in \eqref{chi4} is bounded by $C\|g\|_2^2$, even better than we require. 

Next, using the pointwise bound $\eta_j \leq N$, and Cauchy-Schwarz inequality
$$
\left|\int_{\ST} \varphi_i(\eta_i)\varphi_j(\eta_j)\eta_j{\rm d}\sigma_N\right|  = \left|\int_0^N K\varphi_i(\eta)\varphi_j(\eta)\eta{\rm d}\nu_N\right| \leq N\kappa_2\|\varphi_i\|_2\|\varphi_j\|_2  \leq  \frac{1}{N+1}(\|\varphi_i\|_2^2 + \|\varphi_j\|_2^2 )\ .
$$
Then using Theorem~\ref{CHAOS},  the contribution of the second integral in \eqref{DECLEMC1} to  the sum of integrals in \eqref{chi4} is bounded by $C\|g\|_2^2$, again even better than we require. 

The contribution of the third integral in \eqref{DECLEMC1} has two parts, one involving $\eta_i^2+ \eta_j^2$ and the other involving $\eta_i\eta_j$. 
We can treat the contribution of the part involving $\eta_i^2+ \eta_j^2$ in the same way we estimates the contribution of the second integral in 
\eqref{DECLEMC1}:
$$
\left|\int_{\ST} \varphi_i(\eta_i)\varphi_j(\eta_j)\eta^2_j{\rm d}\sigma_N\right|  = \left|\int_0^N K\varphi_i(\eta)\varphi_j(\eta)\eta^2{\rm d}\nu_N\right| \leq N^2\kappa_2\|\varphi_i\|_2\|\varphi_j\|_2  \leq  \frac{N}{N+1}(\|\varphi_i\|_2^2 + \|\varphi_j\|_2^2 )\ .
$$
Then using Theorem~\ref{CHAOS},  the contribution of this integral in \eqref{DECLEMC1} to the sum of integrals in \eqref{chi4} is bounded by $C\|g\|_2^2$, once more even better than we require.  Finally, we come to the main term of the third integral in \eqref{DECLEMC1}.

Define $\xi_j(\eta_j) = \eta_j\varphi_j(\eta_j) - {\mathbb E}[\eta_j\varphi_j(\eta_j)] = \eta_j\varphi_j(\eta_j) - \langle\eta_j,\varphi_j(\eta_j)\rangle$ so that $\xi_j\perp 1$ in $\mathcal{K}_N$.
By Cauchy-Schwarz inequality, \eqref{Kpoly2} and the fact that $\int_0^N\eta^2{\rm d}\nu_N \leq2$ by \eqref{Kpoly2EF},

\begin{eqnarray*}
\left|\int_{\ST} \varphi_i(\eta_i)\varphi_j(\eta_j)\eta_i\eta_j{\rm d}\sigma_N\right|  &=& \left| \int_{\ST} \xi_i(\eta_i)\xi_j(\eta_j){\rm d}\sigma_N
+ \langle\eta_i,\varphi_i(\eta_i)\rangle\langle\eta_j,\varphi_j(\eta_j)\rangle\right|\\
&\leq & \left|\langle \xi_i, K \xi_j\rangle\right| + 2\|\varphi_i\|_2\|\varphi_j\|_2
\leq\frac{1}{N-1}\|\xi_i\|_2\|\xi_j\|_2 + 2\|\varphi_i\|_2\|\varphi_j\|_2\ .
\end{eqnarray*}
Then since 
${\displaystyle
\|\xi_i\|_2^2=\int_0^N \eta^2\varphi_i(\eta)^2{\rm d}\nu_N-{\mathbb E}[\eta_j\varphi_j(\eta_j)]^2 \leq \int_0^N \eta^2\varphi_i(\eta)^2{\rm d}\nu_N \leq N^2\|\varphi_i\|_2^2}$,
Young's inequality yields
$$
\left|\int_{\ST} \varphi_i(\eta_i)\varphi_j(\eta_j)\eta_i\eta_j{\rm d}\sigma_N\right|  \leq \left(\frac{N^2}{2(N-1)} + 1\right)(\|\varphi_i\|_2^2 + \|\varphi_j\|_2^2)\ .
$$
Then by Theorem~\ref{CHAOS},  the contribution  to the estimate on the sum of integrals in \eqref{chi4} is bounded by $CN\|g\|_2^2$, as required. 

The proof of \eqref{gv4B} is similar. Proceeding in the same way, we must estimate each of the integrals in \eqref{chi2},  \eqref{chi3} and  \eqref{chi4} with each $\varphi$ replaced by $\psi$, and we must show that each of these is bounded above in magnitude by $CN\|s\|_2^2$. 

Replacing $\varphi$ by $\psi$ in the integral on the left in \eqref{chi2MOD}, we must replace $\kappa_2$ by $|\kappa_1|$ on the right, and since 
$\kappa_2 = \frac{2}{N}|\kappa_1|$, this degrades the estimate by a factor of $N$.  However, we had an extra factor of $N$ here when proving
\eqref{gv4}.  For the same reason, the same arguments we used to estimate the various contributions to the error term in 
\eqref{gv4} in which we used $\kappa_2$ go through to yield the required bounds for \eqref{gv4B}, except now these contributions are of the full size 
$CN\|s\|_2^2$ instead of lower order. 

It remains to estimate $\left|\int_{\ST} \psi_i(\eta_i)\psi_j(\eta_j)\eta_i\eta_j{\rm d}\sigma_N\right|$.
Recalling \eqref{sDEFB}, \eqref{Kpoly2EF}, and \eqref{SUPBND1}, by the Cauchy-Schwarz inequality we get
$$
\left|\int_{\ST} \psi_i(\eta_i)\psi_j(\eta_j)\eta_i\eta_j{\rm d}\sigma_N\right|  \leq 2|\alpha_i||\alpha_j| \left(\int_0^N  (\eta-1)^4{\rm d}\nu_N\right)^{1/2}
\left(\int_0^N  \eta^4{\rm d}\nu_N\right)^{1/2}  \leq C(\alpha_i^2+\alpha_j^2)\ .
$$
Multiplying by $\frac{2}{(N-2)(N-1)}$, and summing over $i$, $j$, and $k$, we find that the integral in \eqref{chi4} with each $\varphi$ replaces by $\psi$ is bounded above in magnitude by  $C\|s\|_2^2$, this time a lower order term.
\end{proof}

\section{Improved bound for large $N$}\label{LARGEN}

Using our decomposition for $f$ orthogonal to the constants, we expand
\begin{eqnarray} 
 \Dto(f,f) 
  &=& \Dto(g,g) +  \Dto(s,s) + \Dto(h,h)\label{ORDECA}\\
  &+& 2\Dto(p,h) + 2\Dto(g,s).\label{ORDECB}
  \end{eqnarray}
  
  The next lemma says for sufficiently large $N$, the  two terms in \eqref{ORDECB} are negligibly small. After this is proved, we are left with the relatively simple problem of bounding from below each of the three terms in \eqref{ORDECA}.

 \begin{lm}\label{almorth}  
 There is a constant $C$ independent of $N$ such that if $f\in \cH_N$ is orthogonal to the constants, then the components $s$, $g$, and $h$ in the trial function decomposition \eqref{pdecomp} $f = s+g+h$  satisfy
  \begin{equation}\label{wabnd17}
 2|\Dto(p,h)| + 2|\Dto(g,s)| \leq \frac{C}{N^{3/2}}\|f\|_2^2\ .
\end{equation}
 \end{lm}

 \begin{proof}
Since $\widetilde{P}^{(\gamma)} h = 0$, 
${\displaystyle \Dto(p,h)  = \int_{\ST} \widetilde W^{({\gamma})} p h\,\dd \sigma_N =   -\frac{1-\gamma}{(N-1)^2}\frac1N  \sum_{k=1}^N  \int_{\ST} \eta_k^2 p h\, {\rm d}\sigma_N }$.

Therefore, writing $p = \sum_{j=1}^N \rho_j(\eta_j)$ as in \eqref{struc}, 

\begin{eqnarray}
 \Dto(p,h)  
&= &  -\frac{1-\gamma}{(N-1)^2}\frac1N \sum_{k=1}^N   \int_{\ST} \eta_k^2 \rho_k(\eta_k)h\, \dd \sigma_N\label{nct1}\\
&- & \frac{1-\gamma}{(N-1)^2}\frac1N   \sum_{j\neq k}   \int_{\ST}   \eta_k^2 \rho_j(\eta_j) h\, \dd \sigma_N  \label{nct2}\ .
\end{eqnarray}
The integral in \eqref{nct1} vanishes since $P_k h =0$.
Next, if $j\neq k$ by Cauchy-Schwarz inequality and Lemma~\ref{SUPBND1},
\begin{eqnarray*}
\left| \int_{\ST}   \eta_k^2 \rho_j(\eta_j) h \dd \sigma_N\right| &\leq& \|h\|_2 \left( \int_{\ST}   \eta_k^4 \rho^2_j(\eta_j)\dd \sigma_N\right)^{1/2}\\
&=& \|h\|_2 \left( \int_{\ST}  {\mathbb E}\{ \eta_k^4 | \eta_j\}\rho^2_j(\eta_j)\dd \sigma_N\right)^{1/2}\\
&\leq& \sqrt{24}\|h\|_2\|\rho_j\|_2\ .
\end{eqnarray*}
Therefore, the quantity in \eqref{nct2} is bounded above in absolute value by
$$
\frac{(1-\gamma)\sqrt{24}}{N(N-1)}   \|h\|_2\sum_{j=1}^N \|\rho_j\|_2 \leq 
\frac{(1-\gamma)\sqrt{24}}{(N-1)N^{1/2}} \|h\|_2\left(\sum_{j=1}^N \|\rho_j\|_2^2\right)^{1/2}\ .
$$
Then by Theorem~\ref{CHAOS} and the arithmetic-geometric mean inequality, there is a constant $C$ independent of $N$ and $\gamma\in (0,1)$ such that 
\begin{equation}\label{ORDECAZ1}
\left |\Dto(p,h)\right| \leq C\frac{1}{N^{3/2}}(1-\gamma) (\|p\|_2^2 + \|h\|_2^2) = C\frac{1}{N^{3/2}}(1-\gamma)\|f\|_2^2.
\end{equation}

Next, consider 
${\displaystyle
\Dto(g,s)  = \int_{\ST} \widetilde W^{({\gamma})} gs\,\dd \sigma_N  - \langle g, \widetilde{P}^{(\gamma)} s\rangle}$.
We estimate the terms separately, starting with $\langle g, \widetilde{P}^{(\gamma)} s\rangle$.
Then by Lemma~\ref{SEIG},
\begin{align*}
&\widetilde{P}^{(\gamma)}s = \frac{1}{N}\sum_{k=1}^N m^{(\gamma)}(\eta_k) P_k s 
=\frac{1}{N-1}\sum_{k=1}^N m^{(\gamma)}(\eta_k) \alpha_k \phi_1 (\eta_k)\\
&=\frac{1}{N-1}\sum_{k=1}^N (m^{(\gamma)}(\eta_k)-1)\alpha_k  \phi_1(\eta_k) + \frac{1}{N-1}s\ .
\end{align*}
Therefore, since $g$ and $s$ are orthogonal, 
${\displaystyle
\langle g, P^{(\gamma)}s\rangle = \frac{1}{N-1}\sum_{k=1}^N \langle g, (m^{(\gamma)}-1)\alpha_k  \phi_1 \rangle}$.
Then by H\"older's inequality,
$$
|\langle g, (m^{(\gamma)}-1)\alpha_k  \phi_1 \rangle| \leq |\alpha_k|\|g\|_2 \|m^{(\gamma)}-1\|_4\|\phi_1\|_4\ .
$$

By \eqref{WELOBND}
$m^{(\gamma)}(\eta) -1  =   \frac{\gamma}{N-1}\phi_1(\eta)
-\frac{1-\gamma}{(N-1)^2} \phi_1^2(\eta)
$.
Therefore  for some $C$ independent of $N$ by Lemma~\ref{SUPBND1}
$$\|m^{(\gamma)}-1\|_4  \leq \frac{\gamma}{N-1}\|\phi_1\|_4 
+\frac{1-\gamma}{(N-1)^2} \|\phi_1^2\|_4 \leq \frac{C}{N} $$
$$
|\langle g,\widetilde{P}^{(\gamma)}s\rangle| \leq \frac{C}{(N-1)N} \|g\|_2 \sum_{k=1}^N |\alpha_k| \leq 
\frac{C}{(N-1)N^{1/2}} \|g\|_2 \left(\sum_{k=1}^N \alpha^2_k\right)^{1/2}\ .
$$
By Theorem~\ref{CHAOS}, for some constant $C$ independent of $N$, $\sum_{k=1}^N \alpha^2_k \leq C\|s\|^2$. 
Then by the arithmetic-geometric mean inequality, there is a constant $C$ independent of $N\geq  3$ such that
\begin{equation}\label{ORDECAZ2}
\langle g, \widetilde{P}^{(\gamma)} s\rangle \leq  C\frac{1}{N^{3/2}} (\|g\|^2 + \|s\|^2) \leq  \frac{C}{N^{3/2}} \|f\|^2 \ .
\end{equation}

It remains to consider   
$$\int_{\ST} \widetilde W^{({\gamma})} gs\dd \sigma_N  = -\frac{1-\gamma}{N(N-1)^2} \sum_{k=1}^N \int_{\ST}   (\eta_k^2- 1)  gs {\rm d}\sigma_N 
= -\frac{1-\gamma}{(N-1)^2} \int_{\ST}  F  gs {\rm d}\sigma_N$$
where we have used \eqref{wabnd1} and $F := \frac1N   \sum_{k=1}^N   (\eta_k^2- 1)$.
Again by Lemma~\ref{SUPBND1} there is a positive constant $C$ such that $\|F\|_4 \leq C$ uniformly in $N$.
Then by H\"older's inequality and then Lemma~\ref{sL4B},
$$
\left| \int_{\ST}  F  gs {\rm d}\sigma_N\right|  \leq \|g\|_2\|F\|_4\|s\|_4 \leq CN^{1/2}\|g\|_2\|s\|_2.
$$
Altogether,
$$
\left|\int_{\ST} \widetilde W^{({\gamma})} gs\dd \sigma_N\right| \leq \frac{C}{N^{3/2}}(\|g\|_2^2+ \|s\|_2^2) \leq \frac{C}{N^{3/2}}\|f\|_2^2\ .
$$
Combining this with \eqref{ORDECAZ2} yields the second inequality in \eqref{wabnd17}.
\end{proof}

Now we turn to the ``on diagonal'' lower bounds, first proving some upper bounds on terms involving $\widetilde{ P}^{(\gamma)}$ as defined in \eqref{wabnd2}. 

\begin{lm}\label{HHBND} Let $f\in \cH_N$ be orthogonal to the constants. Let $f = s+g+h$ be its trial function decomposition as in \eqref{pdecomp}.
Then
\begin{equation}\label{HHBND1} 
\Dto(h,h)  \geq \|h\|_2^2 \left(1 -\frac{\gamma-1}{N-1}\right)\ .
\end{equation}
\end{lm}
  
\begin{proof} By Lemma~\ref{nullalpha}, $\widetilde{P}^{(\gamma)}h =0$.  Hence 
${\displaystyle
\Dto(h,h) = \int_{\ST} \widetilde W^{({\gamma})} h^2\dd \sigma_N \geq \|h\|_2^2 \left(1 -\frac{\gamma-1}{N-1}\right)}$
 by \eqref{wabnd1}. 
\end{proof}

\begin{lm}\label{GGBND} Let $f\in \cH_N$ be orthogonal to the constants. Let $f = s+g+h$ be its trial function decomposition as in \eqref{pdecomp}. 
Then there is a constant $C$ independent of $N$ such that 
\begin{equation}\label{GGBND1} 
\Dto(g,g)  \geq \|g\|_2^2 \left(1 -\frac{1}{N} - \frac{C}{N^2}\right) 
\ .
\end{equation}
\end{lm}

\begin{proof} While the pointwise lower bound on $\widetilde W^{({\gamma})}$ sufficed in Lemma~\ref{HHBND}, this is because $\langle h, P^{(\gamma)}h\rangle =0$. Since, in general,  $\langle g, P^{(\gamma)}g\rangle \neq 0$, we need a better lower bound on $\int_{\ST} \widetilde W^{({\gamma})} g^2\dd \sigma_N$.
Combining the {\em identity} in \eqref{wabnd1} with Lemma~\ref{gv4lem},
 \begin{eqnarray}\int_{\ST} \widetilde W^{({\gamma})} g^2\dd \sigma_N &=&
 \|g\|_2^2 - \frac{1-\gamma}{(N-1)^2}\frac1N  \sum_{k=1}^N\int_{\ST}   \eta_k^2 g^2\dd \sigma_N \nonumber\\
 &\geq& \|g\|_2^2  - \frac{1-\gamma}{(N-1)^2}\frac1N  \sum_{k=1}^N\int_{\ST}   \eta_k^2 \varphi_k(\eta_k)^2{\rm d}\sigma_N -  \frac{C}{N^2}\|g\|_2^2\ . 
 \label{GGBND2} 
 \end{eqnarray}
 Combining this with Lemma~\ref{gPUB},
 \begin{equation} \label{GGBND2R} 
 \Dto(g,g)  \geq \|g\|_2^2 -  \frac1N  \sum_{k=1}^N\int_{\ST} \left[ \frac{1-\gamma}{(N-1)^2}\eta_k^2  +   m_N^{(\gamma)}(\eta_k)\right]
 |\varphi_k(\eta_k)|^2{\rm d}\sigma_N - \frac{C}{N^2}\|g\|_2^2\ .
 \end{equation}
 Introducing the variable $y := \frac {\eta_k}{N-1}$, so that $0 \leq y \leq \frac{N}{N-1}$ and recalling\eqref{WELOBND},
\begin{eqnarray*}
 \left[ \frac{1-\gamma}{(N-1)^2}\eta_k^2  +   m_N^{(\gamma)}(\eta_k)\right] 
 &=& (1-\gamma)y^2 + 1 + \gamma\left(\frac{1}{N-1} -y\right) - (1-\gamma)\left(\frac{1}{N-1} -y\right)^2 \\
 &=&  1 + \frac{\gamma}{N-1} - (1-\gamma)\frac{1}{(N-1)^2} + \left((1-\gamma)\frac{2}{N-1} - \gamma\right)y\ .
 \end{eqnarray*}
    We claim that 
    \begin{align*}
    1 + \frac{\gamma}{N-1} - (1-\gamma)\frac{1}{(N-1)^2} + \left((1-\gamma)\frac{2}{N-1} - \gamma\right)y
    \leq 1+\frac{6}{N}\ ,
    \end{align*}
    for any $\gamma\in[0,1]$ and $N\geq 2$.
    In fact, if $(1-\gamma)\frac{2}{N-1} - \gamma\leq 0$, then
        \begin{align*}
    1 + \frac{\gamma}{N-1} - (1-\gamma)\frac{1}{(N-1)^2} + \left((1-\gamma)\frac{2}{N-1} - \gamma\right)y
    \leq 1 + \frac{\gamma}{N-1}
    \leq 1+\frac{2}{N}\ ,
    \end{align*}
    otherwise
            \begin{align*}
    &1 + \frac{\gamma}{N-1} - (1-\gamma)\frac{1}{(N-1)^2} + \left((1-\gamma)\frac{2}{N-1} - \gamma\right)y\\
    &\leq 1 + \frac{\gamma}{N-1} - (1-\gamma)\frac{1}{(N-1)^2} + \left((1-\gamma)\frac{2}{N-1} - \gamma\right)\frac{N}{N-1}\\
    &=(1-\gamma)\left(1+\frac{2}{N-1}+\frac{1}{(N-1)^2}\right)
     \leq 1+\frac{6}{N}\ .
    \end{align*}
 For here is follows that for all $N\geq 2$ 
 \begin{equation} \label{GGBND2S} 
  \left[ \frac{1-\gamma}{(N-1)^2}\eta_k^2  +   m_N^{(\gamma)}(\eta_k)\right] \leq 1 + \frac{6}{N}\ .
 \end{equation}
 
  Therefore, \eqref{GGBND2R} becomes
   \begin{equation} \label{GGBND2RB} 
 \Dto(g,g)  \geq \|g\|_2^2 -  \frac1N  \sum_{k=1}^N\int_{\ST} \left( 1 + \frac{6}{N}\right)
 |\varphi_k(\eta_k)|^2{\rm d}\sigma_N - \frac{C}{N^2}\|g\|_2^2\ .
 \end{equation}
 By Theorem~\ref{CHAOS}, $\sum_{k=1}^N\|\varphi_k\|_2^2 \leq \left(1- \frac2N\right)^{-1}\|g\|_2^2  \leq \left(1+\frac{2}{N-2}\right)\|g\|_2^2$,  so that finally
 \eqref{GGBND1} follows from \eqref{GGBND2RB}. 
\end{proof}

\begin{lm}\label{SSBND} There is a finite constant $C$ independent of $N$ such that if $f\in \cH_N$ is orthogonal to the constants with the trial function decomposition \eqref{pdecomp} $f = s+g+h$,
\begin{equation}\label{SSBND1} 
\Dto(s,s)  \geq \|s\|_2^2 \left(1 -\frac{1}{N} - \frac{C}{N^2}\right) 
\ .
\end{equation}
\end{lm}

\begin{proof} The estimates for  $\Dto(g,g)$ and $\Dto(s,s)$ are of the same form, and therefore, as the proof of Lemma~\ref{GGBND}, we obtain 
\eqref{GGBND2} with $g$ replaced by $s$ (and hence $\varphi_k$ replaced by $\psi_k$), and no other changes. 
However, the bound \eqref{sPUB1} on $\langle s, \widetilde{P}^{(\gamma)}s\rangle$ is worse than the bound on $\langle g, \widetilde{P}^{(\gamma)}g\rangle$ provided
by Lemma~\ref{gPUB} by a factor of $\left(\frac{N}{N-1}\right)^2$.  Altogether we then have, for $\gamma\in (0,1)$ so that \eqref{GGBND2S}  can be used for sufficiently large $N$, 
\begin{eqnarray*}
 \Dto(s,s)  &\geq& \|s\|_2^2 -  \frac1N  \sum_{k=1}^N\int_{\ST} \left[ \frac{1-\gamma}{(N-1)^2}\eta_k^2  +  \left(\frac{N}{N-1}\right)^2 m_N^{(\gamma)}(\eta_k) \right]
 |\psi_k(\eta_k)|^2{\rm d}\sigma_N - \frac{C}{N^2}\|g\|_2^2\\
 &\geq& \|s\|_2^2 -  \frac1N  \sum_{k=1}^N\int_{\ST} \left(\frac{N}{N-1}\right)^2\left[ \frac{1-\gamma}{(N-1)^2}\eta_k^2  +  m_N^{(\gamma)}(\eta_k) \right]
 |\psi_k(\eta_k)|^2{\rm d}\sigma_N - \frac{C}{N^2}\|g\|_2^2\\
 &\geq& \|s\|_2^2 -  \left(\frac{N}{N-1}\right)^{2}   \frac1N  \left( 1 + \frac{6}{N}\right)\sum_{k=1}^N \|\psi_k\|_2^2 - \frac{C}{N^2}\|g\|_2^2\ .
 \end{eqnarray*}
From here the proof proceeds as in the previous lemma.
\end{proof}

\bigskip
We are finally ready to prove Theorem~\ref{CONJMAIN2}, and thus to complete the proof of Theorem~\ref{CONJMAIN3}. 

\begin{proof}[Proof of Theorem~\ref{CONJMAIN2}] Combining Lemma~\ref{almorth} with  Lemmas~\ref{HHBND}, \ref{GGBND} and \ref{SSBND}, we have 
$$
\Dto(f,f)  \geq \left(1 - \frac1N - \frac{C}{N^{3/2}}\right)\|f\|_2^2.
$$
\end{proof}

\section{A broader class of models}

The fact that  with $T_N$ given by \eqref{TNDEF}, the push-forward of $\nu_N\otimes \sigma_{N-1}$ under $T_N$, is $\sigma_N$ underlies all of the 
chaoticity bounds derived in Section~\ref{CHABNDS}. This is true not only for the flat Dirichlet measure, but for a broader class of probability measures on $\STE$
that we now determine. 

The flat Dirichlet measure on $\STE$ arises by equipping $\R_+^N$ with the probability measure $\prod_{j=1}^N f(\eta_j){\rm d\eta_j}$ and then conditioning on
$\sum_{j=1}^N\eta_j =NE$ for the choice $f(\eta) = e^{-\eta}$.

Now, let $f$ be any continuously differentiable and strictly positive probability density on $\R_+$. For $\bset\in \R_+^N$, define $F(\bset) = \prod_{j=1}^N f(\eta_j)$ which is a 
probability density on $\R_+^N$.  Let $\nu_{f,N}$ denote the corresponding probability measure on $\R_+^N$. 

For $E>0$, define 
$\sigma_{f,N,E}$ to be the conditional distribution of $\nu_{f,N}$ on $\STE$. That is, for continuous $\varphi$ on $\R_+^N$, 
$$
{\mathbb E}_{\nu_{f,N}}\left\{ \varphi(\bset)\ \Bigg|\  \frac1N\sum_{j=1}^N \eta_j = E\right\} = \int_{\STE}\varphi(\bset){\rm d}\sigma_{f,N,E}\ .
$$

If $E = \int_{\R_+}\eta f(\eta){\rm d}\eta$, the mean of $f$, then the probability measure is very concentrated near $\STE$; various precise formulations are known under the rubric of ``equivalence of ensembles''.  Because of this,  $\sigma_{f,N,E}$ is approximately a product distribution, and the various coordinate functions $\eta_1,\dots,\eta_N$ are approximately independent random variables under  $\sigma_{f,N,E}$.
 
 For the symmetric Dirichlet distributions, and in particular, the flat Dirichlet distribution, this near independence is valid for all values of $E$, not only the mean of $f$, because of a scaling property that characterizes the Gamma distributions, from which they arise.

The different ``slices''  $\STE$ of $\R_+^N$, all simplexes, are related by scaling: For $t>0$ define $S_t:\R_+^N\to \R_+^N$ by 
$S_t(\bset) = t^{-1}\bset$.
Then $S_t$ is invertible from $\STE$ onto $\mathcal{S}_{N,tE}$.  Let $S_t\# \sigma_{N,E}$ denote the push-forward of $\sigma_{N,E}$  under $S_t$. That is, for $\varphi$ on $\mathcal{S}_{N,tE}$.
$$
\int_{\mathcal{S}_{N,tE}}\varphi(\bset){\rm d} S_t\# \sigma_{f,N,E} =
\int_{\STE} \varphi(S_t(\bset)){\rm d}\sigma_{f,N,E}=  \int_{\STE} \varphi(t^{-1}\bset){\rm d}\sigma_{f,N,E}\ .
$$
Recall that the Gamma distribution with parameters $\alpha,\lambda>0$ is the probability on $\R_+$ with density $f_{\alpha,\lambda}(x) :=\frac{\lambda^\alpha}{\Gamma(\alpha)}x^{\alpha-1}e^{-\lambda x}$.  The following theorem gives a characterization of the Gamma distribution that is relevant to the considerations of this paper.

\begin{thm} For all $t>0$, $S_t\# \sigma_{f,N,E} = \sigma_{f,N,tE}$ if and only if $f$ is a Gamma density. 
\end{thm}

\begin{proof} It is clear that $S_t\# \sigma_{f,N,E} = \sigma_{f,N,tE}$ for all $t>0$  if and only if for all $\bset,{\boldsymbol \xi}\in \STE$,
$$
\frac{F(t\bset)}{F({t\boldsymbol \xi})} = \frac{F(\bset)}{F({\boldsymbol \xi})} 
$$
for all $t>0$. Define $h = \log f$. Then we can rewrite this as
${\displaystyle
\sum_{j=1}^N h(t\eta_j) - \sum_{j=1}^N h(t\xi_j) = C}$
for some constant $C$. Differentiating in $t$, and then setting $t=1$,
${\displaystyle
\sum_{j=1}^N(\eta_jh'(\eta_j) - \xi_jh'(\xi_j)) = 0
}$.
Define $\varphi(\eta) = \eta h'(\eta)$. Then $\sum_{j=1}^N \varphi(\eta_j)$ is constant on $\STE$. By the Lemma~\ref{Gamma}, this means that for some constant $a,b$, $\varphi(\eta) = a\eta + b$.

Therefore,  for some constants $a,b$, $h'(\eta) =  a+\frac{b}{\eta}$. It follows that for some $c\in \R$, 
$$
h(\eta) = a\eta + b\log \eta + c\ .
$$
This is the general form of the logarithm of a Gamma density, and hence $f$ is a Gamma density.
\end{proof}

\begin{lm}\label{Gamma} Let $\varphi$ be $C^1$ on $(0,\infty)$ and suppose that $\sum_{j=1}^N \varphi(\eta_j)$ is constant on $\STE$. Then for some $a,b\in \R$, 
$\varphi(\eta) = a\eta + b$. 
\end{lm}

\begin{proof} Define $\psi(\eta) = \varphi(\eta) - \varphi(0)$. For any $\lambda \in (0,1)$,
$\psi(NE) = \psi(\lambda NE) + \psi((1-\lambda)NE)$
so $\psi$ that  is linear.
\end{proof}

This scaling property of the Gamma-Dirichlet measures leads directly to (3.1) in the paper, and is implied by it (for dimension $N-1$). The methods used in this paper can therefore be applied to the models with invariant measures generated by Gamma distributions that are considered in \cite{GKS12} and \cite{S15}, and the results of this section provide another perspective on why this class of models is natural. 

\medskip

\noindent{\Large \bf Acknowledgements}\ \ 
\smallskip

 We thank Domokos Sz\'asz for valuable discussion and for stimulating us to work on this problem, and we thank Pietro Caputo for valuable discussion. EC thanks the Dipartimento di Matematica, Università di Roma ``la Sapienza'', where part of this work was done while he was a visiting Professor, for its hospitality.

\end{document}